\documentclass{amsart}

\usepackage{amsmath,amsfonts,amssymb,amsthm} 
\usepackage[all, cmtip]{xy} 
\usepackage{enumerate} 
\usepackage{pict2e} 
\usepackage{caption} 
\usepackage{tikz,tikz-3dplot}
\usepackage[T1]{fontenc} 

\usepackage{todonotes}

\theoremstyle{definition}
\newtheorem{definition}{Definition}[section]

\newtheorem{question}[definition]{Question}

\theoremstyle{plain}
\newtheorem{lemma}[definition]{Lemma}
\newtheorem{theorem}[definition]{Theorem}
\newtheorem{proposition}[definition]{Proposition}
\newtheorem{corollary}[definition]{Corollary}
\newtheorem{conjecture}[definition]{Conjecture}
\newtheorem{ansatz}[definition]{Ansatz}

\theoremstyle{remark}
\newtheorem{remark}[definition]{Remark}
\newtheorem{example}[definition]{Example}

\newcommand{\AAA}{\mathbb{A}}

\newcommand{\CC}{\mathbb{C}}

\newcommand{\FF}{\mathbb{F}}

\newcommand{\PP}{\mathbb{P}}
\newcommand{\QQ}{\mathbb{Q}}
\newcommand{\RR}{\mathbb{R}}

\newcommand{\ZZ}{\mathbb{Z}}
\newcommand{\WP}{\mathbb{W}\mathbb{P}}

\newcommand{\calL}{\mathcal{L}}

\newcommand{\calO}{\mathcal{O}}
\newcommand{\calP}{\mathcal{P}}

\newcommand{\calV}{\mathcal{V}}

\DeclareMathOperator{\Pic}{Pic}
\DeclareMathOperator{\rank}{rank}
\DeclareMathOperator{\im}{im}
\DeclareMathOperator{\Gr}{Gr}

\numberwithin{table}{section}

\begin{document}
\title[Mirror symmetry, degenerations and fibrations]{Mirror symmetry, Tyurin degenerations and fibrations on Calabi-Yau manifolds}

\author[C. F. Doran]{Charles F. Doran}
\address{Department of Mathematical and Statistical Sciences, 632 CAB, University of Alberta, Edmonton, AB, T6G 2G1, Canada}
\email{charles.doran@ualberta.ca}
\thanks{C. F. Doran would like to acknowledge support from the Natural Sciences and Engineering Research Council
of Canada (NSERC) and the Visiting Campobassi Professorship of the University of Maryland.}

\author[A. Harder]{Andrew Harder}
\address{Department of Mathematical and Statistical Sciences, 632 CAB, University of Alberta, Edmonton, AB, T6G 2G1, Canada}
\email{aharder@ualberta.ca}
\thanks{A. Harder was supported by an NSERC Post-Graduate Scholarship.}

\author[A. Thompson]{Alan Thompson}
\address{Department of Pure Mathematics, University of Waterloo, 200 University Ave West, Waterloo, ON, N2L 3G1, Canada}
\email{am6thomp@uwaterloo.ca}
\thanks{A. Thompson was supported by a Fields-Ontario Postdoctoral Fellowship with funding provided by NSERC and the Ontario Ministry of Training, Colleges and Universities.}
\date{}

\begin{abstract} We investigate a potential relationship between mirror symmetry for Calabi-Yau manifolds and the mirror duality between quasi-Fano varieties and Landau-Ginzburg models. More precisely, we show that if a Calabi-Yau admits a so-called Tyurin degeneration to a union of two Fano varieties, then one should be able to construct a mirror to that Calabi-Yau by gluing together the Landau-Ginzburg models of those two Fano varieties. We provide evidence for this correspondence in a number of different settings, including Batyrev-Borisov mirror symmetry for K3 surfaces and Calabi-Yau threefolds, Dolgachev-Nikulin mirror symmetry for K3 surfaces, and an explicit family of threefolds that are not realized as complete intersections in toric varieties.
\end{abstract}
\maketitle

\section{Introduction}

The aim of this paper is to investigate the relationship between mirror symmetry for Calabi-Yau manifolds and for Fano varieties of the same dimension. 

Classically, mirror symmetry is a network of conjectures relating the properties of two mirror dual Calabi-Yau manifolds. For us, unless otherwise stated, a \emph{Calabi-Yau manifold} will always be a smooth compact K\"{a}hler manifold $V$ with trivial canonical bundle $\omega_V \cong \calO_V$ and vanishing cohomology groups $H^i(V,\calO_V)$ for all $0 < i < \dim(V)$. 

A similar duality has been proposed for Fano varieties. In physics, Eguchi, Hori and Xiong \cite{gqc} postulated that a $d$-dimensional manifold $X$ with $c_1(X) > 0$ should be mirror to a \emph{Landau-Ginzburg model} $(Y,\mathsf{w})$, where $Y$ is a $d$-dimensional K\"ahler manifold and $\mathsf{w}$ is a function $\mathsf{w} \colon Y \to \CC$. This correspondence was then incorporated into the framework of homological mirror symmetry, as a correspondence between the directed Fukaya category \cite{vcm}  (resp. the bounded derived category of singularities \cite{dccstcs}) associated to $(Y,\mathsf{w})$ and the bounded derived category of coherent sheaves on $X$ (resp. the Fukaya category of $X$). More recently, Katzarkov, Kontsevich and Pantev \cite{btttlgm} conjectured that if $X$ is a Fano variety, then the Landau-Ginzburg model $(Y,\mathsf{w})$ of $X$ is in fact a quasi-projective variety that satisfies certain specific conditions and, moreover, that there is a mirror relationship between the Hodge numbers of $X$ and certain Hodge-theoretic invariants of $(Y,\mathsf{w})$.

It is expected that this notion of mirror symmetry for Fano varieties is related to classical mirror symmetry for Calabi-Yau manifolds of one dimension lower, as follows. If $X$ is a $d$-dimensional Fano variety with mirror Landau-Ginzburg model $(Y,\mathsf{w})$, then a general fibre of $\mathsf{w}$ is expected to be a $(d-1)$-dimensional Calabi-Yau variety that is mirror, in the classical sense, to a generic anticanonical hypersurface in $X$. 

This raises a natural question: is mirror symmetry for $d$-dimensional Fano varieties related to classical mirror symmetry for Calabi-Yau manifolds of the same dimension $d$? In this paper we outline a correspondence that provides a potential answer to this question.

This correspondence may be described as follows. Let $V$ be a $d$-dimensional Calabi-Yau manifold and suppose that $V$ admits a degeneration to a union $X_1 \cup_Z X_2$ of two quasi-Fano varieties glued along an anticanonical hypersurface $Z$ (such degenerations are called \emph{Tyurin degenerations}). Then we claim that the mirror $W$ of $V$ admits a fibration by $(d-1)$-dimensional Calabi-Yau manifolds, with general fibre $S$ that is mirror to $Z$. Moreover, $W$ can be constructed topologically by gluing together the Landau-Ginzburg models $(Y_1,\mathsf{w}_1)$ and $(Y_2,\mathsf{w}_2)$ of $X_1$ and $X_2$, in a sense to be made precise in Section \ref{sect:glue}.

The first person to observe traces of such a correspondence was probably Dolgachev \cite{mslpk3s}, who noticed that Dolgachev-Nikulin mirror symmetry for K3 surfaces matches Type II degenerations (of which Tyurin degenerations are a special case) with elliptic fibrations on the mirror. After this, the first mention of a higher dimensional version appears to be due to Tyurin, who gave a brief hint of its existence at the very end of \cite{fvcy}. 

More recently, a variant of the construction presented here was worked out in detail by Auroux \cite{slfmscydc}, in the special case where $V$ is a double cover of a Fano variety $X$ ramified over a smooth member of $|-2K_X|$; this $V$ admits a Tyurin degeneration to the union of $X$ with itself.

\medskip

The structure of this paper is as follows. In Section \ref{sect:setup} we describe our construction. We begin with a $d$-dimensional Calabi-Yau manifold $V$ which admits a Tyurin degeneration to a union $X_1 \cup_Z X_2$ of quasi-Fano varieties glued along an anticanonical hypersurface $Z$. Then we show that the Landau-Ginzburg models $(Y_1,\mathsf{w}_1)$ and $(Y_2,\mathsf{w}_2)$ of $X_1$ and $X_2$ may be glued together to form a new variety $W$, which is fibred by Calabi-Yau $(d-1)$-folds topologically mirror to $Z$, so that the Euler numbers of $V$ and $W$ satisfy the mirror relationship $\chi(V) = (-1)^d\chi(W)$. This suggests that $V$ and $W$ should be thought of as mirror dual. In the threefold case we provide even more evidence for this conjecture: if we make the assumptions that $W$ is Calabi-Yau and that the K3 surface $Z$ is Dolgachev-Nikulin mirror to a general fibre of the fibration on $W$, then we can show that $V$ and $W$ are in fact topologically mirror.

In the remaining sections of the paper we discuss this correspondence in several special cases. In Section \ref{sec:BBduality} we discuss the case of Batyrev-Borisov mirror symmetry for surfaces and threefolds. Indeed, suppose that $V$ is a K3 surface or Calabi-Yau threefold constructed as an anticanonical hypersurface in a Gorenstein toric Fano $3$- or $4$-fold, determined by a reflexive polytope $\Delta$. We show that a nef partition $\Delta_1$, $\Delta_2$ of $\Delta$ determines both a Tyurin degeneration $X_1 \cup_Z X_2$ of $V$ and a fibration $\pi\colon W \to \PP^1$ on a birational model $W$ of its Batyrev mirror, so that the general fibre of $\pi$ is Batyrev-Borisov mirror dual to the intersection $Z = X_1 \cap X_2$. 

Specializing to the threefold case, we further show that the singular fibres of the K3 surface fibration $\pi\colon W \to \PP^1$ contain numerical information about $X_1$ and $X_2$, and describe a relationship between $W$ and the Landau-Ginzburg models of $X_1$ and $X_2$. Unfortunately a corresponding result in the K3 surface case is difficult to prove for combinatorial reasons, but we conjecture the form that such a result should take.

Section \ref{sec:DNforK3s} is concerned with Dolgachev-Nikulin mirror symmetry for K3 surfaces. We revisit Dolgachev's \cite{mslpk3s} mirror correspondence between Type II degenerations (of which Tyurin degenerations are a special case) and elliptic fibrations, which may be thought of as a generalization of the correspondence described in Section \ref{sect:setup}. Consideration of several explicit examples suggests a way to enhance our conjectures to cope with more general Type II degenerations, which may contain more than two components.

In Section \ref{sec:beyond3folds} we discuss how this theory fits with classical mirror symmetry for threefolds. We begin by showing that, if $V$ is a Calabi-Yau threefold that undergoes a Tyurin degeneration (satisfying certain technical conditions), then mirror symmetry predicts the existence of a K3 fibration on the mirror threefold $W$, with properties consistent with those expected from the theory in Section \ref{sect:setup}. Following this, we specialize our discussion to the case of threefolds fibred by quartic mirror K3 surfaces, as studied in \cite{cytfmqk3s}. In this setting we explicitly construct candidate mirror threefolds, along with Tyurin degenerations of them, and show that they have the properties predicted by Section \ref{sect:setup}. In particular, this provides an important illustration of our theory using threefolds that are not complete intersections in toric varieties, thereby giving evidence that the ideas of Section \ref{sect:setup} apply beyond the toric setting of Section \ref{sec:BBduality}.

Finally, Section \ref{sec:noncomm} discusses the limitations of our construction. Indeed, it appears that difficulties arise for Tyurin degenerations of $V$ which occur along loci in moduli that are disjoint from points of maximally unipotent monodromy. In this case, we seem to have no guarantee of the existence of a mirror fibration on $W$; an example where this occurs is given in Example \ref{ex:cubic}. Instead we present evidence that, if $W$ is replaced by its bounded derived category of coherent sheaves $\mathbf{D}^b(W)$, it should be possible to find a non-commutative fibration of $\mathbf{D}^b(W)$ by Calabi-Yau categories, which might be thought of as homologically mirror to the Tyurin degeneration of $V$.

\section{Setup and preliminary results} \label{sect:setup}

Our aim is to provide evidence for a mirror correspondence between a certain type of degeneration of Calabi-Yau manifolds, called a Tyurin degeneration, and Calabi-Yau manifolds constructed by gluing Landau-Ginzburg models. We begin by defining these objects.

\subsection{Smoothing Tyurin degenerations} \label{sec:smoothtyur}

A smooth variety $X$ is called a \emph{quasi-Fano} variety if its anticanonical linear system contains a smooth Calabi-Yau member and $H^i(X,\calO_X) = 0$ for all $i >0$. Given this, a \emph{Tyurin degeneration} is a degeneration $\calV \to \Delta$ of Calabi-Yau manifolds over the unit disc $\Delta \subset \CC$, such that the total space $\calV$ is smooth and the central fibre is a union of two quasi-Fano varieties that meet normally along a smooth variety $Z$, with $Z \in |-K_{X_i}|$ for each $i \in \{1,2\}$. Degenerations of this type have been studied by Lee \cite{leethesis}, who coined the name.

This construction can be reversed, and a family of Calabi-Yau manifolds built up from a pair of quasi-Fano varieties $X_1$ and $X_2$ as follows. Let $Z$ be a smooth variety which is a member of both $|-K_{X_1}|$ and $|-K_{X_2}|$, and suppose that there are ample classes $D_1 \in \Pic(X_1)$ and $D_2 \in \Pic(X_2)$ which both restrict to the same ample class $D \in \Pic(Z)$ (this last condition is needed to ensure that \cite[Theorem 4.2]{ldncvsdcyv}, which gives the existence of a smoothing, can be applied in our setting).  Let $X_1 \cup_Z X_2$ denote the variety which is a normal crossings union of $X_1$ and $X_2$ meeting along $Z$. 

With this setup, we say that $X_1 \cup_Z X_2$ is \emph{smoothable} to a Calabi-Yau manifold $V$ if there exists a complex manifold $\mathcal{V}$ equipped with a map $\psi\colon \mathcal{V}\rightarrow \Delta$ so that the fibre $\psi^{-1}(0) = X_1 \cup_Z X_2$, the fibre $\psi^{-1}(t)$ is a smooth Calabi-Yau manifold for any $t\in \Delta \setminus \{0\}$, and $V$ is a general fibre of $\calV$. It follows from a theorem of Kawamata and Namikawa \cite[Theorem 4.2]{ldncvsdcyv} that $X_1 \cup_Z X_2$ is smoothable to a Calabi-Yau manifold $V$ if and only if $N_{Z/X_1}$ and $N_{Z/X_2}$ are inverses of one another and, moreover, that the resulting manifold $V$ is unique up to deformation.

\subsection{Gluing Landau-Ginzburg models}\label{sect:glue}

Let us first define what we mean by Landau-Ginzburg (LG) model in this paper. In \cite{harderthesis}, a notion of a LG model is defined which conjecturally encapsulates the LG models of Fano varieties, and even goes further to describe the LG models of many quasi-Fano varieties. For general quasi-Fano varieties, however, we do not believe that this definition is sufficient; in particular, it seems that for general quasi-Fanos we must drop any expectation that our LG model be algebraic.

For this reason, in this paper we adopt a much more general definition.

\begin{definition}
A \emph{Landau-Ginzburg (LG) model} of a quasi-Fano variety is a pair $(Y,\mathsf{w})$ consisting of a a K\"ahler manifold $Y$ satisfying $h^1(Y) = 0$ and a proper map $\mathsf{w}\colon Y \rightarrow \mathbb{C}$. The map $\mathsf{w}$ is called the \emph{superpotential}.
\end{definition}

Note that this definition leaves room for the image of $\mathsf{w}$ to be an open set in $\mathbb{C}$. If $Y$ is quasi-projective then the Hodge numbers of such LG models $(Y,\mathsf{w})$ are defined in \cite{btttlgm}; however, in the general case it is unclear how this should be done. Instead, following \cite{btttlgm}, we propose that if $(Y,\mathsf{w})$ is the LG model of a quasi-Fano variety $X$, then we should have
\begin{equation}h^{i}(Y,\mathsf{w}^{-1}(t)) = \sum_{j} h^{d-i+j,j}(X),\label{eqn:LG}\end{equation}
where $h^i(Y,\mathsf{w}^{-1}(t))$ is the rank of the cohomology group of the pair $H^i(Y,\mathsf{w}^{-1}(t))$ and $t$ is a generic point in the image of $\mathsf{w}$. We also expect that if $(Y,\mathsf{w})$ is the LG model of $X$, then the smooth fibres of $\mathsf{w}$ should be mirror to generic anticanonical hypersurfaces in $X$.

With notation as in the previous section, it now seems pertinent to ask whether there is any relationship between the LG models of the quasi-Fano varieties $X_1$ and $X_2$, and mirror symmetry for $V$. Indeed, it seems natural to expect that these LG models could be somehow glued together to give a mirror $W$ for $V$, since we are, in a topological sense, gluing $X_1$ and $X_2$ together to form $V$ (see \cite{fvcy} for details on this topological construction).

In more detail, we expect that if $Y_i$ is the LG model of $X_i$, equipped with superpotential $\mathsf{w}_i$, then the monodromy symplectomorphism on $\mathsf{w}_i^{-1}(t)$ (for $t$ a regular value of $\mathsf{w}_i$) associated to a small loop around $\infty$  can be identified under mirror symmetry with the restriction of  the Serre functor of the bounded derived category of coherent sheaves $\mathbf{D}^b(X_i)$ on $X_i$ to the bounded derived category of coherent sheaves $\mathbf{D}^b(Z)$ on $Z$ \cite{mavcm,btttlgm}. This Serre functor is simply $(-)\otimes \omega_{X_i}[d]$ where $[d]$ denotes shift by $d = \dim X_i$. Thus, up to a choice of shift, we see that the action of monodromy on $\mathsf{w}_i^{-1}(t)$ should be identified with the autoequivalence of $\mathbf{D}^b(Z)$ induced by taking the tensor product with $\omega_{X_i}|_Z = N_{Z/X_i}^{-1}$. 

Now recall that if $X_1 \cup_Z X_2$ is smoothable to $V$, then we have $N_{Z/X_1} \otimes N_{Z/X_2} = \mathcal{O}_Z$, so the monodromy symplectomorphism $\phi_1$ associated to a clockwise loop around infinity on $\mathsf{w}_1^{-1}(t)$ should be same as the monodromy $\phi_2^{-1}$ associated to a \emph{counter-clockwise} loop around infinity on $\mathsf{w}_2^{-1}(t)$. It should be noted that, for this to make sense, we must assume that the fibres of $\mathsf{w}_1$ and $\mathsf{w}_2$ are topologically the same Calabi-Yau manifold, which we denote by $S$; this assumption is stronger than the assumption that both are mirror to $Z$. 

Now we glue these LG models as follows. For each $i \in \{1,2\}$, choose $r_i$ so that $|\lambda| \leq r_i$ for every $\lambda$ in the critical locus of $\mathsf{w}_i$. Then choose local trivializations of $Y_i$ over $U_{i} = \{ z \in \mathbb{C} : |z| > r_i\}$ and let $Q_{i} = \mathsf{w}_i^{-1}(U_{i})$. This local trivialization is topologically equivalent to expressing $Q_i$ as a gluing of the ends of $B_i=S \times [-1,1]\times (-1,1)$ together via the map 
\[\widetilde{\phi}_i \colon p \times \{-1\} \times (z) \longmapsto  \phi_i(p) \times \{1\} \times (z),\]
where $\phi_i$ is the monodromy symplectomorphism, and we identify $S \times \{-1\} \times (-1,1)$ with $S \times \{1\} \times (-1,1)$. 

\begin{figure}
\begin{center}
\begin{tikzpicture}[scale = 0.75]

\draw (4,0) arc (0:180: 2cm);
\draw[dashed] (4,0) arc (0:180: 2cm and 0.6cm);
\draw (0,0) arc (180:360: 2cm and 0.6cm);

\draw (0,-2) arc (180:360: 2cm);
\draw (2,-2) ellipse (2cm and 0.6cm);


\draw [->] (2,3.2) -- (2,2.2);
\draw [->] (2,-5.2) -- (2,-4.2);


\node at (2,3.6) {{\small $Y_1$}};
\node at (2,-5.6) {$Y_2$};

\node at (2.5,2.8) {$\mathsf{w}_1$};
\node at (2.5,-4.8) {$\mathsf{w}_2$};


\draw (7,-2) ellipse( 2cm and 0.6cm);
\draw (5,-2.5) arc (180:360: 2cm and 0.6cm);
\draw (8.75,-2.25) arc (30:150: 2cm and 0.6cm);
\draw (7,0.5) ellipse( 2cm and 0.6cm);
\draw (5,0) arc (180:360: 2cm and 0.6cm);
\draw (8.75,0.25) arc (30:150: 2cm and 0.6cm);
\draw (9,0)--(9,0.5);
\draw (5,0)--(5,0.5);
\draw (9,-2)--(9,-2.5);
\draw (5,-2)--(5,-2.5);


\draw [->] (7,3.2) -- (7,1.5);
\draw [->] (7,-5.2) -- (7,-3.5);

\node at (7,3.6) {$Y_1|_{B_1}$};
\node at (7,-5.6) {$Y_2|_{B_2}$};

\node at (7.8,2.3) {$\mathsf{w}_1|_{B_1}$};
\node at (7.8,-4.3) {$\mathsf{w}_2|_{B_2}$};

\node at (4.5,3.6) {$\supset$};
\node at (4.5,-5.6) {$\supset$};

\node at (7,-1) {$\reflectbox{\rotatebox[origin=c]{90}{$\simeq$}}$ diffeo};

\node at (12,-4) {Identify $B_1$ and $B_2$};


\draw (12,-1) circle (2cm);
\draw[dashed] (14,-1.2) arc (0:180: 2cm and 0.6cm);
\draw[dashed] (14,-0.8) arc (0:180: 2cm and 0.6cm);
\draw (10,-1.2) arc (180:360: 2cm and 0.6cm);
\draw (10,-0.8) arc (180:360: 2cm and 0.6cm);


\draw[->] (12,3.2) -- (12,1.2);


\node at (12,3.6) {$W$};
\node at (12.3,2.3) {$\pi$};

\end{tikzpicture}
\end{center}
\caption{Gluing $Y_1$ and $Y_2$ to give $W$}
\label{fig:gluing}
\end{figure}
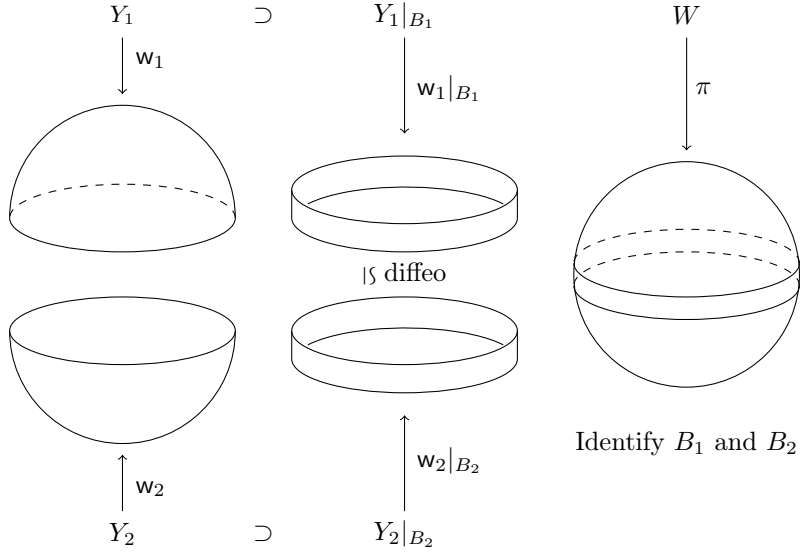

Assuming that $\phi_1 = \phi^{-1}_2$ (which, we recall, conjecturally follows from smoothability of $X_1 \cup_Z X_2$), we can identify $B_1$ with $B_2$ by the map
\[\tau\colon p \times [x] \times (z) \longmapsto p \times [-x] \times (-z).\]
Under this identification of $B_1$ and $B_2$, it is clear that  $\tau \cdot \widetilde{\phi}_1 = \widetilde{\phi}_2$. Thus the identification $\tau$ gives an isomorphism between $Q_1$ and $Q_2$, allowing us to glue $Y_1$ to $Y_2$ along $Q_1$ and $Q_2$ to produce a $C^\infty$ manifold $W$. This gluing respects the fibrations $\mathsf{w}_1$ and $\mathsf{w}_2$, so $W$ is equipped with a fibration $\pi$ over the gluing of $\mathbb{C}$ with $\mathbb{C}$ described above. It is clear that the base of this fibration is just the 2-sphere $\mathrm{S}^2$. This procedure is illustrated in Figure \ref{fig:gluing}.

\begin{example}
As a sanity check, we can perform this construction with elliptic curves. Take a degeneration of an elliptic curve to a union of two copies of $\mathbb{P}^1$ meeting in two points (Kodaira type ${I}_2$). The LG model of $\mathbb{P}^1$ is the map from $\CC^{\times}$, which is topologically a twice-punctured rational curve, to $\CC$ given by 
\[\mathsf{w}\colon x \longmapsto x + \dfrac{1}{x}.\]
This map $\mathsf{w}$ is a double covering of $\AAA^1$ ramified at two points. One can check that monodromy of this fibration around the point at infinity is trivial. Let $Y_1$ and $Y_2$ be copies of this LG model of $\mathbb{P}^1$. Then we may glue $Y_1$ and $Y_2$ as described above. The resulting topological space is a double cover of $S^2$ which is ramified at four points. This is simply the 2-dimensional torus, which is topologically mirror to the original elliptic curve
\end{example}

\begin{theorem}\label{thm:glue-euler}
Let $X_1$ and $X_2$ be $d$-dimensional quasi-Fano varieties which contain the same anticanonical Calabi-Yau hypersurface $Z$, such that $K_{X_1}|_Z + K_{X_2}|_Z = 0$. Let $(Y_1,\mathsf{w_1})$ and $(Y_2, \mathsf{w_2})$ be Landau-Ginzburg models of $X_1$ and $X_2$, and suppose that the fibres of $\mathsf{w_1}$ and $\mathsf{w_2}$ are topologically the same Calabi-Yau manifold, which is topologically mirror to $Z$. Finally, let $V$ be a Calabi-Yau variety obtained from $X_1 \cup_Z X_2$ by smoothing and let $W$ be the variety obtained by gluing $Y_1$ to $Y_2$ as above. Then
\[\chi(V) = (-1)^d\chi(W),\]
where $\chi$ denotes the Euler number.
\end{theorem}
\begin{proof}
Start by recalling the long exact sequence of the pair $(Y_i,\mathsf{w}_i^{-1}(t))$, for $t$ a regular value of $\mathsf{w}_i$,
\[\cdots \rightarrow  H^n(Y_i,\mathbb{C}) \rightarrow  H^n(\mathsf{w}_i^{-1}(t),\mathbb{C}) \rightarrow  H^{n+1}(Y_i,\mathsf{w}_i^{-1}(t);\mathbb{C}) \rightarrow  H^{n+1}(Y_i,\mathbb{C}) \rightarrow \cdots.\]
Since Euler numbers are additive in long exact sequences, we have that $\chi(Y_i) = \chi(Y_i,\mathsf{w}_i^{-1}(t)) + \chi(\mathsf{w}_i^{-1}(t))$. By Equation \eqref{eqn:LG}, we see that $\chi(Y_i,\mathsf{w}^{-1}(t))$ is equal to $(-1)^d\chi(X_i)$, where $d$ is the dimension of $Y_i$. Thus $\chi(Y_i) = (-1)^d\chi(X_i) + \chi(\mathsf{w}_i^{-1}(t)))$. Moreover, since $\mathsf{w}^{-1}(t)$ is topologically mirror to $Z$ by assumption, we have that $\chi(Z) = (-1)^{d-1}\chi(\mathsf{w}^{-1}(t))$, which gives $\chi(Y_i) = (-1)^d(\chi(X_i) - \chi(Z))$.

On the other hand, the Mayer-Vietoris exact sequence
\[\cdots \rightarrow  H^n(W,\mathbb{C}) \rightarrow  H^n(Y_1,\mathbb{C}) \oplus  H^n(Y_2,\mathbb{C}) \rightarrow  H^n(Y_1 \cap Y_2,\mathbb{C}) \rightarrow \cdots\]
gives $\chi(W) = \chi(Y_1) + \chi(Y_2) -\chi(Y_1 \cap Y_2)$. Since $Y_1\cap Y_2$ is a fibration over an annulus, we can compute its cohomology using the Wang sequence \cite[Theorem 11.33]{mhs}
\[\cdots\rightarrow  H^n(Y_1 \cap Y_2,\mathbb{C}) \rightarrow  H^n(\mathsf{w}^{-1}(t),\mathbb{C}) \xrightarrow{T_{n}-\mathrm{Id}}  H^n(\mathsf{w}^{-1}(t),\mathbb{C}) \rightarrow \cdots, \]
where $T_n$ is the action of monodromy on $H^n(\mathsf{w}^{-1}(t),\mathbb{C})$ associated to a small loop around our annulus, to obtain $\chi(Y_1 \cap Y_2) = 0$. Putting everything together, we obtain
\[\chi(W) = (-1)^d(\chi(X_1) + \chi(X_2) - 2\chi(Z)).\]

Finally, since $X_1 \cup_Z X_2$ is smoothable to $V$, we can compute the Euler characteristic of $V$ by applying \cite[Proposition IV.6]{leethesis}, which states that
\[\chi(V) = \chi(X_1) + \chi(X_2) - 2 \chi(Z).\]
We therefore have that $\chi(W) = (-1)^d\chi(V)$, as claimed.
\end{proof}

This is precisely the relationship between the Euler characteristics of mirror dual Calabi-Yau varieties. In the next subsection, we will provide more evidence for the hypothesis that $W$ is the mirror dual of the original Calabi-Yau variety $V$, in the special case where $V$ is a Calabi-Yau threefold.

\begin{remark}
Note that the requirement that there exist two ample divisors $D_1$ and $D_2$, on $X_1$ and $X_2$ respectively, which restrict to the same divisor on $Z$ was not used at all in the proof of Theorem \ref{thm:glue-euler}. Moreover, despite the fact that the proof of \cite[Theorem 4.2]{ldncvsdcyv} uses this assumption in a material way (in order to prove the pro-representability of the log deformation functor), the topological construction of the gluing of $X_1$ and $X_2$ can be performed without it. 

For instance, let us take a generic K3 surface $Z$ with Picard lattice of rank $2$ isomorphic to the lattice with Gram matrix
\[\left( \begin{matrix} 4 & 6 \\ 6 & 6 \end{matrix} \right).\]
Such a K3 surface embeds into both $\mathbb{P}^3$ and the intersection of a quadric $Q$ and a cubic $C$ in $\mathbb{P}^5$. Let us blow up $\mathbb{P}^3$ in $Z \cap Z'$ for some generic K3 surface $Z'$ in $\mathbb{P}^3$, calling the result $X_1$, and blow up $Q \cap C$ in the intersection of $Z$ and a generic hyperplane section in $\mathbb{P}^5$, calling the result $X_2$. Then the normal crossings variety $X_1 \cup_Z X_2$ is not K\"ahler, so we cannot find $D_1$ and $D_2$ as above. However, both $V$ and $W$ can be constructed, as $C^\infty$ manifolds,  from $X_1 \cup_Z X_2$ by the method we have described. We wonder whether $V$ and $W$ represent a mirror pair of non-K\"ahler Calabi-Yau manifolds.
\end{remark}

\subsection{The threefold case}\label{sec:threefold}

With notation as before, Lee \cite{cycsncv} has computed the Hodge numbers of $V$ in the case where $X_1$ and $X_2$ are smooth threefolds. Let us define $\rho_i\colon  H^2(X_i,\mathbb{Q}) \rightarrow  H^2(Z,\mathbb{Q})$ for $i=1,2$ to be the restriction and define $k = \rank (\im(\rho_1) + \im(\rho_2))$.

\begin{theorem}\label{thm:lee}
\cite[Corollary 8.2]{cycsncv} Let $V$ be a Calabi-Yau threefold constructed as as smoothing of $X_1 \cup_Z X_2$, as above. Then
\begin{align*}
 h^{1,1}(V) &= h^{2}(X_1) + h^2(X_2) - k - 1,\\
 h^{2,1}(V) &= 21 + h^{2,1}(X_1) + h^{2,1}(X_2) - k.
\end{align*}
\end{theorem}

On the other side of the picture, we have a corresponding result for $W$.

\begin{proposition}\label{prop:threefolds}
Let $W$ be as above and let $S$ be a general fibre of the map $\pi$. Assume that $\dim W = \dim S +1 =3$. Then 
\[h^2(W) =1+ h^2(Y_1,S) + h^2(Y_2,S) + \ell,\]
where $\ell$ is the rank of the subgroup of $ H^2(S,\mathbb{C})$ spanned by the intersection of the images of $ H^2(Y_1,\mathbb{C})$ and $ H^2(Y_2,\mathbb{C})$ under the natural restriction maps.
\end{proposition}
\begin{proof} Let $U$ be the annulus along which $B_1$ and $B_2$ are glued, and let $Q = \pi^{-1}(U)$ be its preimage in $W$. We begin by computing the rank of $ H^2(W,\mathbb{C})$ using the Mayer-Vietoris sequence
\[\cdots \rightarrow  H^1(Q,\mathbb{C}) \rightarrow  H^2(W,\mathbb{C}) \rightarrow  H^2(Y_1,\mathbb{C}) \oplus  H^2(Y_2,\mathbb{C}) \xrightarrow{r_1^Q-r_2^Q}  H^2(Q,\mathbb{C}) \rightarrow \cdots,\]
where $r_i^Q$ are the natural restriction maps from $ H^2(Y_i,\mathbb{C})$ to $ H^2(Q,\mathbb{C})$. From the Wang sequence, we obtain $ H^1(Q,\mathbb{C}) = \mathbb{C}$. So, using the assumption that $ H^1(Y_1,\mathbb{C}) = H^1(Y_2,\mathbb{C}) = 0$, we see that $ H^2(W,\mathbb{C})$ is isomorphic to the direct product of $\mathbb{C}$ and the kernel of the restriction map $r_1^Q - r_2^Q$.  We note that this map fits into a commutative triangle
\begin{equation*}
\xymatrix{
 H^2(Y_1,\mathbb{C}) \oplus  H^2(Y_2,\mathbb{C}) \ar[rd]_{r^S_1-r^S_2} \ar[r]^(0.65){r_1^Q - r_2^Q} &  H^2(Q,\mathbb{C}) \ar[d]^{r_Q^S}\\
            &  H^2(S,\mathbb{C}).}
\end{equation*}
Now, since $S$ is a K3 surface, we have $h^1(S) =0$, and it follows from the Wang sequence that the map $r_Q^S$ is injective. Thus the kernel of ${r_1^S-r_2^S}$ is the same as the kernel of ${r_1^Q-r_2^Q}$. Elementary linear algebra gives that the rank of this kernel is $h^2(Y_1) + h^2(Y_2) - \rank(\im (r_1^S)+\im(r_2^S))$. So we obtain
\[h^2(W) = 1 + h^2(Y_1) + h^2(Y_2) - \rank(\im (r_1^S)+\im(r_2^S)).\]

Now, for $i = 1,2$ we have exact sequences 
\[0 \longrightarrow  H^2(Y_i,S;\mathbb{C}) \longrightarrow  H^2(Y_i,\mathbb{C})  \stackrel{r_i^S}{\longrightarrow}  H^2(S,\mathbb{C}) \longrightarrow \cdots.\]
Which give
\[h^2(Y_i) = h^2(Y_i,S) +  \rank(\im (r_i^S)).\]
Putting together with the previous expression, the proposition follows.
\end{proof}

Therefore, if $W$ admits a complex structure for which it is Calabi-Yau, then we compute
\begin{align*}
\chi(W) &=  2h^{1,1}(W) - 2h^{2,1}(W) \\
&= 2(1+h^{2}(Y_1,S) + h^{2}(Y_2,S)  + \ell) - 2h^{2,1}(W).
\end{align*}
Equation \eqref{eqn:LG} then gives that $h^2(Y_i,S) = h^{2,1}(X_i)$, so
\[\chi(W) = 2(h^{2,1}(X_1) + h^{2,1}(X_2) - h^{2,1}(W) + \ell + 1).\]

Furthermore, from Theorems \ref{thm:glue-euler} and \ref{thm:lee}, we also know that 
\begin{align*}
\chi(W) &= -\chi(V) \\
&= -2h^{1,1}(V) + 2h^{2,1}(V) \\
&= -2(h^2(X_1) + h^2(X_2) - k - 1) + 2(h^{2,1}(X_1) + h^{2,1}(X_2) +  21 - k) \\
&= 2(h^{2,1}(X_1) + h^{2,1}(X_2) - h^2(X_1) - h^2(X_2) + 22)
\end{align*}

Putting this together, we have that $h^{2,1}(W) = \ell - 21 + h^{2}(X_1) + h^{2}(X_2)$. So in order for $W$ and $V$ to be topologically mirror to one another, we must have
\[\ell- 21 + h^{2}(X_1) + h^{2}(X_2) =  h^{2}(X_1) + h^{2}(X_2) -k - 1,\]
which is equivalent to $\ell +k = 20$. This is true if $S$ and $Z$ are mirror dual in the sense of Dolgachev-Nikulin, given the lattice polarization on $Z$ (resp. $S$) coming from the sum of the images of the restriction maps $ H^2(X_i,\mathbb{Z}) \rightarrow  H^2(Z,\mathbb{Z})$ (resp. the intersection of the images of the restriction maps $ H^2(Y_i,\mathbb{Z}) \rightarrow  H^2(S,\mathbb{Z})$). Thus mirror symmetry for $V$ and $W$ is consistent with mirror symmetry for $S$ and $Z$.

\begin{remark} \label{rem:threefoldphilosophy} In the case where $W$ is Calabi-Yau and $S$ and $Z$ are Dolgachev-Nikulin mirror dual, the expressions 
\begin{align*}
h^{2,1}(V) &= 21 + h^{2,1}(X_1) + h^{2,1}(X_2) - k \\
h^{1,1}(W) &=1+ h^2(Y_1,S) + h^2(Y_2,S) + \ell
\end{align*}
could be thought of as mirror dual decompositions of the corresponding Hodge numbers, in the following sense.

The Hodge number $h^{2,1}(X_i)$ may be interpreted as the fibre dimension of the natural map from the moduli space of pairs $(X_i,Z)$ to the moduli space of appropriately polarized K3 surfaces $Z$ (which has dimension $20-k$). Thus the degenerate fibre $X_1 \cup_Z X_2$ should have $h^{2,1}(X_1) + h^{2,1}(X_2) + 20 - k = h^{2,1}(V) -1$ deformations, and such Tyurin degenerations should appear in codimension $1$ in the moduli space of $V$. 

We thus obtain a decomposition of $h^{2,1}(V)$ into contributions $h^{2,1}(X_i)$ coming from deformations of each $X_i$, a contribution $(20-k)$ from deformations of the gluing locus $Z$, and $1$ for the codimension in the moduli space. 

On the mirror side a similar statement holds: $h^{1,1}(W)$ can be decomposed into contributions $h^{2}(Y_i,S)$ coming from the LG-models $(Y_i,\mathsf{w}_i)$ (these will be interpreted later as counts of components in singular fibres), a contribution $\ell$ from divisors on the generic fibre $S$, and $1$ for the class of a general fibre (compare \cite[Lemma 3.2]{cytfmqk3s}). The picture is completed by noting that $h^2(Y_i,S) = h^{2,1}(X_i)$ and $\ell = 20-k$.
\end{remark}

\section{Batyrev-Borisov mirror symmetry}\label{sec:BBduality}

In this section we will prove a number of results that illustrate the situation considered in the previous section in the special case of Batyrev-Borisov mirror symmetry. For background on the definitions and concepts used in this section, we refer the reader to \cite{msag,tv}. However, since our conventions differ very slightly from those used in the references above, before we proceed we will briefly outline the notation to be used in the remainder of this section.

Let $M$ be a free $\mathbb{Z}$-module of rank $d$, let $\Delta$ be a reflexive polytope in $M \otimes \mathbb{R} = M_\mathbb{R}$, and denote the boundary of $\Delta$ by $\partial \Delta$. Let $N = \mathrm{Hom}(M,\mathbb{Z})$ be the dual lattice to $M$ and denote by $\langle\cdot, \cdot \rangle$ the natural bilinear pairing from $N \times M$ to $\mathbb{Z}$. Let 
\[\Delta^\circ = \{ u \in N_\mathbb{R} : \langle u, v \rangle \geq -1 \text{ for all } v \in \Delta\}\] 
denote the polar polytope to $\Delta$.

Let $\mathbb{P}_\Delta$ be the $d$-dimensional toric variety associated to the polytope $\Delta$. The toric variety $\mathbb{P}_\Delta$ is Fano and has at worst Gorenstein singularities. Following \cite[Theorem 2.2.24]{dpmscyhtv}, one can find a toric variety $X_\Delta$ which is a toric partial resolution of singularities of $\mathbb{P}_\Delta$ and which has at worst Gorenstein terminal singularities. Such $X_{\Delta}$ is referred to as a \emph{maximal projective crepant partial (mpcp) resolution of singularities} of $\mathbb{P}_\Delta$. In the future, we shall fix one such $X_\Delta$ for any given $\mathbb{P}_\Delta$. The variety $X_\Delta$ can be presented as a quotient of some Zariski open subset $U \subseteq \mathbb{C}^{|\partial\Delta \cap M|}$ by the torus $(\mathbb{C}^\times)^{|\partial \Delta|-d}$. There is thus a homogeneous coordinate ring $\mathbb{C}[\{ z_\rho \}_{\rho \in \partial \Delta \cap M}]$ on $X_\Delta$. 

The vanishing of each coordinate $z_\rho$ determines a divisor on $X_\Delta$, invariant under the natural action of the torus $(\mathbb{C}^\times)^d$, which we call $D_\rho$. The anticanonical divisor $-K_{X_\Delta}$ of $X_\Delta$ is linearly equivalent to $\sum_{\rho \in \partial \Delta \cap M}D_\rho$, and the cycle class group $A_1(X_\Delta)$ is generated by the divisors $D_\rho$. A divisor $\sum_{\rho \in \partial \Delta \cap M}b_\rho D_\rho$ for $b_\rho \in \mathbb{Z}$ is Cartier if and only if there is a piecewise linear function $\varphi$ on $M_\mathbb{R}$, which takes integral values on $M$ and which is linear on the cones of the fan defining $X_\Delta$, so that $\varphi(\rho) = b_\rho$ for all $\rho$.

A \emph{nef partition} of $\Delta$ is a partition of $\partial \Delta \cap M$ into sets $E_1,\dots, E_k$, so that for each $i =1,\dots, k$, the divisor $\sum_{\rho \in E_i} D_\rho$ is nef and Cartier. Let us denote the line bundle thus associated to $E_i$ by $\mathcal{L}_i$. We will let $\Delta_i = \mathrm{Conv}(E_i \cup \{0_M\})$; in a mild abuse of terminology, we also refer to $\Delta_1,\ldots,\Delta_k$ as a nef partition of $\Delta$.

Batyrev's \cite{dpmscyhtv} toric version of mirror symmetry claims that the generic anticanonical hypersurfaces in $X_\Delta$ and $X_{\Delta^\circ}$ are mirror dual. Moreover, if we have a nef partition of $\Delta$, then the complete intersection $V$ of generic sections of $\mathcal{L}_1,\dots, \mathcal{L}_k$ is again Calabi-Yau. Borisov \cite{tmscycigtfv} and Batyrev-Borisov \cite{cycitv} propose that there is a similar combinatorial construction of the mirror of $V$. In this case, we define
\[\nabla_i =\left\{ u \in N_\mathbb{R} : \begin{array}{l} \langle u, v \rangle \geq 0 \ \text{ for all } v \in E_j, j \neq i \\ \langle u, v \rangle \geq -1 \text{ for all } v \in E_i \end{array}\right\}\]
and let $\nabla = \mathrm{Conv}( \nabla_1\cup \dots \cup \nabla_k)$. This is a reflexive polytope and $\nabla_1,\ldots,\nabla_k$ is a nef partition of $\nabla$. The complete intersection $W$ in $X_\nabla$ cut out by generic sections of the line bundles associated to $\nabla_1,\dots, \nabla_k$ is a Calabi-Yau variety, which is expected to be mirror dual to $V$. 

Finally, a \emph{refinement} of a nef partition $E_1,\ldots,E_k$ is defined to be another nef partition $F_1,\ldots, F_{k+1}$ so that $F_i = E_i$ for $1 \leq i \leq k-1$ and $E_k = F_k \cup F_{k+1}$. 

Now, let $X_\Delta$ be a $d$-dimensional toric variety as above. Suppose that $V$ is a Calabi-Yau complete intersection of nef divisors in $X_{\Delta}$, determined by a nef partition $E_1,\dots, E_k$. Our aim is to show that, if $F_1,\dots, F_{k+1}$ is a refinement of $E_1,\dots, E_k$, then this combinatorial data determines
\begin{itemize}
\item a Tyurin degeneration of $V$, and
\item a pencil of quasi-smooth varieties birational to Calabi-Yau $(d-k-1)$-folds inside of the Batyrev-Borisov mirror $W$. 
\end{itemize}
 In the case where $V$ is a threefold, we show that this pencil induces a K3 surface fibration on some birational model of $W$ and that the singular fibres of this fibration carry information about the Tyurin degeneration of $V$. We will then compare this with the LG model picture in the previous section.

\subsection{Tyurin degenerations}\label{sect:Tyur} More precisely, let $\mathcal{L}_i$ be the line bundles on $X_\Delta$ associated to the $E_i$. The refinement $F_1,\ldots,F_{k+1}$ gives rise to a pair of nef line bundles $\mathcal{L}'_k$ and $\mathcal{L}'_{k+1}$ so that $\mathcal{L}'_k \otimes \mathcal{L}_{k+1}' = \mathcal{L}_k$. Let $s_i \in  H^0(X_\Delta,\mathcal{L}_i)$ be generic sections determining a quasi-smooth Calabi-Yau complete intersection $V$ in $X_\Delta$. If we let $s_k'$ and $s_{k+1}'$ be sections of $\mathcal{L}_k'$ and $\mathcal{L}_{k+1}'$ respectively, then $s_k' s_{k+1}'$ is a section of $\mathcal{L}_k$. 

We can use this to construct a pencil of complete intersections as follows. First, let
\[V' = \cap_{i=1}^{k-1} \{s_i = 0\}\]
and assume that $V'$ is connected and quasi-smooth; it is also clear that $V'$ is quasi-Fano. Then take the pencil
\[\mathcal{Q}\colon \{ts_{k} - s'_k s'_{k+1} = 0\} \cap V'\]
in $\mathbb{A}^1 \times X_\Delta$, with $t$ a parameter on $\mathbb{A}^1$. If we assume that $X_\Delta$ is a smooth resolution of $\mathbb{P}_\Delta$, then the only singularities of $\mathcal{Q}$ in a neighbourhood of $0 \in \mathbb{A}^1$ are along $t=s_k = s_{k}'=s_{k+1}' = 0$, which we call $\Sigma$. 

Note that since $\Sigma$ is the intersection of a set of nef divisors in $X_\Delta$, it has no base locus and its singularities are contained in the singular set of $X_\Delta$. Furthermore, the intersection of $\Sigma$ with any torus invariant subvariety of $X_\Delta$ is irreducible, thus $\Sigma$ itself is either irreducible or a union of non-intersecting subvarieties of $X_\Delta$. As a result, if $X_\Delta$ is smooth, then so is $\Sigma$, for general enough choices of sections. 

We can thus resolve the singularities of $\mathcal{Q}$ by blowing up $t=s_k'=0$  inside of $\mathbb{A}^1 \times X_\Delta$ and taking the proper transform of $\mathcal{Q}$. The result is a Tyurin degeneration of $V$, so that the fibre over $0$ of the degeneration is equal to the union of $\hat{X}_1$ and $X_2$, where $\hat{X}_1$ is a quasi-Fano variety given by blowing up $X_1 := V' \cap \{s_k' = 0\}$ along $V' \cap \{s_k = s_k' = s_{k+1}'=0\}$ and $X_2$ is a quasi-Fano variety given by $X_2 := V' \cap \{s_{k+1}'=0\}$.

 In the general situation, where $X_\Delta$ is not a smooth resolution of $\mathbb{P}_\Delta$, we can still perform all of the steps above, but we will have singularities occurring at every step in general. The resulting degeneration will not be a Tyurin degeneration in the strict sense, but should still include data corresponding to the quasi-Fano varieties $X_1,X_2$ and the blown up locus $\Sigma$. We note here that a version of the smoothability result of Kawamata and Namikawa that works for mildly singular varieties has been explored in the thesis of Lee \cite{leethesis}. The singular case may also be interpreted as equipping the union of $X_1$ and $X_2$ with a log structure (see e.g. \cite{lgm} and the references therein), which accounts for the subvariety $\Sigma$ and determines the smoothing to $V$.

\subsection{Pencils and fibrations on the mirror} \label{subsec:mirrorpencils} Now we will look at how this nef partition is reflected in the mirror. For ease of notation we restrict ourselves to the case where $V$ is a hypersurface; all of the results below generalize in the obvious way to refinements of $k$-partite nef partitions corresponding to codimension $k$ complete intersections.

Since $V$ is a hypersurface, the nef partition $E_1 = \Delta$ is trivial. Let $\Delta_1,\Delta_2$ denote the polytopes corresponding to the refinement $F_1,F_2$ of $E_1$; we thus have $\Delta = \mathrm{Conv}(\Delta_1 \cup \Delta_2)$.

Now, the Batyrev dual of $V$ is a Calabi-Yau variety $W$ embedded as an anticanonical hypersurface in $X_{\Delta^\circ}$. By definition, $W$ is cut out by an equation in the homogeneous coordinate ring of $X_{\Delta^{\circ}}$, which may be written as
\[f:=\sum_{\rho \in \Delta \cap M} a_\rho \prod_{\sigma \in \partial \Delta^\circ \cap N} z_\sigma^{\langle \sigma, \rho\rangle + 1}  =0,\]
where $a_{\rho}$ are generically chosen complex coefficients. We will take a pencil $\mathcal{P}$ of hypersurfaces in $W$, for $[s:t] \in \mathbb{P}^1$,  defined by the intersection of $W$ with hypersurfaces of the form
\[s \sum_{\rho \in \Delta_1 \cap M\setminus 0_M} a_\rho \prod_{\sigma \in \partial \Delta^\circ \cap N} z_\sigma^{\langle \sigma, \rho\rangle +1} = t a_0 \prod_{\sigma \in \partial\Delta^\circ \cap N} z_\sigma.\]
Note that, away from $[s:t]=[0:1]$, this pencil may also be defined by the pair of equations
\begin{align*}
s \sum_{\rho \in \Delta_1 \cap M\setminus 0_M} a_\rho \prod_{\sigma \in \partial \Delta^\circ \cap N} z_\sigma^{\langle \sigma, \rho\rangle +1} &= t a_0 \prod_{\sigma \in \partial\Delta^\circ \cap N} z_\sigma \\
s \sum_{\rho \in \Delta_{2} \cap M\setminus 0_M} a_\rho \prod_{\sigma \in \partial \Delta^\circ \cap N} z_\sigma^{\langle \sigma, \rho\rangle +1} &= (s-t)a_0\prod_{\sigma \in \partial\Delta^\circ \cap N} z_\sigma.
\end{align*}
as the sum of these two equations is just $s f$, which vanishes precisely along $W$ if $s \neq 0$.

We wish to show that the pencil $\calP$ induces a fibration on $W$. If we can show that a general member of this pencil is Calabi-Yau, then this will follow from:

\begin{proposition}\label{prop:fibration}
If $S\subseteq W$ is a smooth Calabi-Yau $(d-1)$-fold in a $d$-dimensional Calabi-Yau manifold, then the linear system $|S|$ is base-point free and hence there is a map $\pi\colon W \rightarrow \mathbb{P}^1$ with $S$ as a fibre.
\end{proposition}
\begin{proof}
By adjunction, $\mathcal{O}_S(S) = N_{S/W} = \omega_S =\mathcal{O}_S$. Let $\iota \colon S \hookrightarrow W$ be the inclusion map, then we have a short exact sequence of sheaves
\[0 \longrightarrow \mathcal{O}_W \stackrel{s}{\longrightarrow} \mathcal{O}_W(S) \longrightarrow \iota^*\mathcal{O}_S \longrightarrow 0,\]
where $s$ is a section of $\mathcal{O}_W(S)$ whose vanishing locus is $S$. This short exact sequence gives rise to a long exact sequence in cohomology
\[0 \longrightarrow  H^0(W,\mathcal{O}_W) \longrightarrow H^0(W, \mathcal{O}_W(S)) \longrightarrow  H^0(W,\iota^*\mathcal{O}_S) \longrightarrow 0,\]
where the vanishing of $H^1(W,\calO_W)$ follows from the Calabi-Yau property of $W$. From this sequence, we see that the restriction of a generic section of $\mathcal{O}_W(S)$ to $S$ is a nonzero section of $\iota^*\mathcal{O}_S$, which is \emph{a fortiori} non-vanishing. Thus $|S|$ is base-point free and determines a map $W \to \mathbb{P}^1$, since $h^0(\mathcal{O}_W(S))=2$ by the exact sequence above.
\end{proof}

To apply this proposition, we need to show that a general member $S$ of the pencil $\calP$ is Calabi-Yau. When $\dim(W) = 2$ or $3$, this will follow from the next proposition, which is proved in \cite{harderthesis}.

\begin{proposition} \cite{harderthesis}
If $\dim(W)=2$, then a general member of the pencil $\calP$ is a smooth elliptic curve. If $\dim(W) = 3$, then a general member of the pencil $\calP$ is a smooth blow-up of a K3 surface $S$ and, moreover, $S$ is Batyrev-Borisov dual to the intersection of the quasi-Fano varieties $X_1$ and $X_2$ from Section \ref{sect:Tyur}.
\end{proposition}

In the case where $\dim(W) = 2$, it therefore follows immediately from Proposition \ref{prop:fibration} that  the pencil $\mathcal{P}$ is an elliptic fibration on $W$. However, if $W$ is a Calabi-Yau threefold and a general member of $\calP$ is a blown up K3 surface $\hat{S}$, then we will need to get rid of the $(-1)$-curves in $S$ before we can apply Proposition \ref{prop:fibration}. We will do this by performing a series of birational transformations.

\begin{lemma}
Let $W$ be a Calabi-Yau threefold and let $\hat{S}$ be a blown up K3 surface in $W$. If $C$ is a $(-1)$-curve in $\hat{S}$, then $N_{C/W} \cong \mathcal{O}_{\mathbb{P}^1}(-1) \oplus \mathcal{O}_{\mathbb{P}^1}(-1)$.
\end{lemma}
\begin{proof}
We have a short exact sequence of sheaves on $C$,
\[0 \longrightarrow \Theta_C \longrightarrow \Theta_W|_C \longrightarrow N_{C/W} \longrightarrow 0.\]
Since $c_1(\Theta_C) = 2$ and $c_1(\Theta_W|_C) =c_1(\Theta_W)|_C = 0$, it follows that $N_{C/W} \cong \mathcal{O}_{\mathbb{P}^1}(a) \oplus \mathcal{O}_{\mathbb{P}^1}(b)$ for some $a,b \in \ZZ$ with $a + b = -2$ (see, for example, \cite[Section 1]{rccyt}). 

We may embed the normal bundle $N_{C/\hat{S}}$ into $N_{C/W}$ to get a short exact sequence of line bundles
\[0 \longrightarrow N_{C/\hat{S}} \longrightarrow N_{C/W} \longrightarrow \mathcal{L} \longrightarrow 0\]
for some line bundle $\mathcal{L}$. Since $C$ is a $(-1)$-curve in $\hat{S}$, we know that $N_{C/\hat{S}} \cong \mathcal{O}_{\mathbb{P}^1}(-1)$. Furthermore, we have that $c_1(\mathcal{L}) = -1$ from the fact that $c_1(N_{C/W}) = -2$, thus $\mathcal{L} \cong \mathcal{O}_{\mathbb{P}^1}(-1)$. The long exact sequence in cohomology coming from the above short exact sequence proves that $ H^0(C,N_{C/W}) = 0$ and therefore we must have that $N_{C/W} = \mathcal{O}_{\mathbb{P}^1}(-1) \oplus \mathcal{O}_{\mathbb{P}^1}(-1)$.
\end{proof}

  Therefore, any $(-1)$-curve $C$ in $\hat{S}$ may be blown up to produce a a variety $\widetilde{W}$ with exceptional divisor a copy of $\mathbb{P}^1 \times \mathbb{P}^1$. This copy of $\mathbb{P}^1 \times \mathbb{P}^1$ can be smoothly contracted along either ruling. Contracting along one ruling recovers $W$; we denote the variety obtained by contracting along the other ruling by $W^+$. The strict transform of $\hat{S}$ in $\widetilde{W}$ is just $\hat{S}$ itself, but the contraction $\hat{W} \to W^+$ contracts the $(-1)$-curve $C$ in $\hat{S}$. This is an example of a birational operation called a \emph{flop}.

Repeating this for all $(-1)$ curves in $\hat{S}$, we obtain a birational model of $W$ which is a smooth Calabi-Yau threefold in which $\hat{S}$ has been contracted to its minimal model, which is a K3 surface. Call the Calabi-Yau threefold resulting from this process $\hat{W}$. Applying Proposition \ref{prop:fibration} to $\hat{W}$, we see that we have proved:

\begin{theorem}
Let $W$ be a Calabi-Yau threefold containing a smooth blown up K3 surface $\hat{S}$. By performing a sequence of flops on $W$, we may obtain a birational model $\hat{W}$ of $W$ which admits a fibration $\pi\colon \hat{W} \rightarrow \mathbb{P}^1$, so that the minimal model $S$ of $\hat{S}$ is a general fibre of $\pi$.
\end{theorem}

Putting everything together, in the case of toric hypersurface Calabi-Yau threefolds we obtain:

\begin{corollary}\label{cor:mirsym}
If $\Delta$ is a $4$-dimensional reflexive polytope that admits a nef partition $\Delta_1,\Delta_2$ and $V$ is a general anticanonical hypersurface in $X_\Delta$, then $V$ admits a Tyurin degeneration and its Batyrev dual $W$ has a birational model which admits a fibration by K3 surfaces. Moreover, the general fibre in this fibration on $W$ is Batyrev-Borisov mirror dual to the complete intersection K3 surface in $X_{\Delta}$ determined by $\Delta_1,\Delta_2$.
\end{corollary}

\begin{remark}
There seems to be an inherent incompatibility between Batyrev (and Batyrev-Borisov) duality and K3 surface fibrations on Calabi-Yau threefolds, since one can show that even in very basic examples, there are exceptional curves in $\hat{S}$ that cannot be avoided by simply changing the birational model of the toric ambient space $X_{\Delta^\circ}$. A notable exception occurs when either $\Delta_1$ or $\Delta_2$ is $1$-dimensional. In this case, one of the two component quasi-Fano varieties to which $V$ degenerates is itself a toric variety. This is mirrored by the fact that $X_{\Delta^\circ}$ admits a morphism to $\mathbb{P}^1$ which induces the required K3 fibration on $W$. These seem to be a subset of the ``toric fibrations'' which have been studied extensively in the physics literature by a number of authors (\cite{sfk3f,ak3fpip,wmetk3hsc} to name a few).
\end{remark}

\begin{remark}
One can consider more general refinements of nef partitions, by taking a nef partition $F_1,\dots, F_\ell$ so that for each $E_i$, there is a subset $I_i$ of $\{1,\dots, \ell\}$ so that $E_i = \cup_{j\in I_i} F_j$. These will give rise to generalized degenerations of the Calabi-Yau $V$ determined by $E_1,\dots, E_k$ to unions of quasi-Fano varieties, and families of Calabi-Yau varieties of codimension $\ell-k$ in $W$. The issue here, of course, is that it is hard to prove that this gives a fibration on a birational model of $W$. Despite this, these families of Calabi-Yau varieties surely have properties related to the LG models of the appropriate quasi-Fano varieties.
\end{remark}

\begin{remark} \label{rem:higherdim}
Generalizing Corollary \ref{cor:mirsym} to higher dimensions seems to be a challenge, since we have made use of both the minimal model program for surfaces and a characterization of flops in three dimensions. Of course, both of these objects have analogues in higher dimensions, but they are much more oblique and not likely to be useful in such a general situation. As it stands, the results above suffice to show that we obtain rational maps from $W$ to $\mathbb{P}^1$ corresponding to every bipartite nef partition of $\Delta$.
\end{remark}

\subsection{Singular fibres and a comparison with LG models} 

Now we will analyze the singular members of the pencil $\mathcal{P}$, which are birational to the fibres of $\hat{W}$. Using this, we can give a very accurate description of the singular fibres of $\hat{W}$, up to birational transformations. As we shall see, the resulting theory meshes nicely with the LG model picture described in Section \ref{sect:setup}. The results in this section can be extended quite generally, but for simplicity we will restrict ourselves to the situation of threefold hypersurfaces. 

With notation as in the previous section, the nef partition $\Delta_1, \Delta_2$ of $\Delta$ determines a pair of polytopes $\nabla_1, \nabla_2 \subseteq \Delta^{\circ}$, which are a nef partition of $\nabla :=\mathrm{Conv}(\nabla_1 \cup \nabla_2) \subseteq \Delta^{\circ}$. Note that this inclusion may be strict; it therefore does not follow that we have a degeneration of $W$ to quasi-Fano varieties dual to $X_1$ and $X_2$. 

Our first goal is to look at the member of the pencil $\calP$ over $[s:t]=[0:1]$ and understand its meaning in terms of the Tyurin degeneration described in Section \ref{sect:Tyur}. 

\begin{proposition}\label{prop:fibinf}
The member of the pencil $\mathcal{P}$ corresponding to $[s:t]=[0:1]$ is equal to 
\[\bigcup_{\sigma \in (\Delta^\circ \setminus \nabla)\cap N} (D_\sigma \cap W).\]
In other words, the linear system defining the pencil $\calP$ is associated to the line bundle $\mathcal{O}_W(\sum_{\sigma \in (\Delta^\circ \setminus \nabla)\cap N} D_\sigma|_W)$.
\end{proposition} 

\begin{proof}
Recall that the pencil $\mathcal{P}$ is defined as the intersection of $W$ with hypersurfaces of the form
\[s \sum_{\rho \in \Delta_1 \cap M\setminus 0_M} a_\rho \prod_{\sigma \in \partial \Delta^\circ \cap N} z_\sigma^{\langle \sigma, \rho\rangle +1} - t a_0 \prod_{\sigma \in \partial\Delta^\circ \cap N} z_\sigma = 0.\]
Thus, at least on an open set of $W$, we may write this family of hypersurfaces as a rational map from $X_{\Delta^\circ}$ to $\mathbb{P}^1$, defined by $\phi\colon [z_\sigma] \mapsto [s:t]$ where
\begin{align*}
s &= \prod_{\sigma \in (\Delta^\circ\setminus \nabla) \cap N} z_\sigma \\
t &= \Bigg(\sum_{\rho \in \Delta_1 \cap M\setminus 0_M} a_\rho \prod_{\sigma \in \partial \Delta^\circ \cap N} z_\sigma^{\langle \sigma, \rho\rangle +1}\Bigg) \Bigg/ \Bigg(\prod_{\sigma \in \partial\nabla \cap N } z_{\sigma}\Bigg) 
\end{align*}

We want to show that this map is defined on $W$ away from the base locus of $\mathcal{P}$. Note that homogeneity away from $\sum_{\sigma \in \partial \nabla \cap N}D_\sigma$ is clear, since both terms are sections of $\mathcal{L} :=\mathcal{O}_{X_{\Delta^\circ}}(\sum_{\sigma \in (\Delta^\circ \setminus \nabla) \cap N}D_\sigma)$. Now, if $\sigma \in \nabla_2$, then the numerator in the expression for $t$ above has a factor of $z_\sigma$, since $\langle \sigma, \rho \rangle \geq 0$ for all $\rho \in \Delta_1$ and $\sigma \in \nabla_2$. Thus both terms are also sections of $\calL$ along $D_\sigma$ for $\sigma \in \nabla_2$. Moreover, on the restriction of $\phi$ to $W$, we notice that the expression for $t$ can also be written as 
\[\Bigg(\sum_{\rho \in \Delta_2 \cap M} a_\rho \prod_{\sigma \in \partial \Delta^\circ \cap N} z_\sigma^{\langle \sigma, \rho\rangle +1}\Bigg) \Bigg/ \Bigg(\prod_{\sigma \in \partial\nabla \cap N } z_\sigma\Bigg).\]
and thus, for the same reason as above, both terms are sections of $\mathcal{L}$ along $D_{\sigma} \cap W$ for $\sigma \in \nabla_1$. Finally, \cite[Proposition 6.3]{cycitv} implies that $\nabla \cap N = (\nabla_1 \cap N) \cup (\nabla_2 \cap N)$, so we deduce that the expressions for $s$ and $t$ above form honest global sections of $\mathcal{L}|_W$. Thus the map $\phi$ is well-defined away from the base locus of $\mathcal{L}|_W$ and the fibre of $\phi$ over $s=0$ is as required.
\end{proof}

  It follows from the proof of Proposition \ref{prop:fibinf} that the line bundle $\mathcal{O}_W(\hat{S})$ is just $\mathcal{O}_{W}(\sum_{\sigma \in (\Delta^\circ\setminus\nabla)\cap N }D_\sigma)$. Since $W$ is an anticanonical hypersurface, the intersection of a divisor $D_\sigma$ with $W$ is empty if and only if $\sigma$ lies in the relative interior of a facet of $\Delta^\circ$. If $\sigma$ is in the interior of a codimension $2$ face of $\Delta^\circ$, then a calculation analogous to that performed in \cite[\S 3.3]{lptk3s} shows that $D_\rho \cap W$ has $1+\ell^*(\Gamma(\sigma)) \ell^*(\Gamma(\sigma)^\vee)$ irreducible components, where $\Gamma(\sigma)$ is the smallest face of $\Delta^\circ$ containing $\sigma$, $\Gamma(\sigma)^\vee$ is the face of $\Delta$ made up of points $\rho$ satisfying $\langle \sigma, \rho \rangle = -1$, and $\ell^*(\Gamma)$ denotes the number of lattice points in the relative interior of $\Gamma$.  Finally, if $\sigma$ lies in a codimension $\geq 3$ face of $\Delta^\circ$, then $D_\rho \cap W$ is irreducible for generic $W$.

\begin{proposition}
If $\sigma$ is in $(\Delta^\circ \setminus \nabla) \cap N$, then $D_\rho \cap W$ has a single irreducible component. Therefore, the member of the pencil $\calP$ corresponding to $[s:t]=[0:1]$ has 
\[\# (\Delta^\circ \setminus \nabla) \cap N\]
irreducible components.
\end{proposition}

\begin{proof}
First, if $\sigma$ is contained on the relative interior of a facet of $\Delta^\circ$, then $\Gamma(\sigma)^\vee$ is a single vertex $\eta$ of $\Delta$. Without loss of generality, we can assume that $\eta \in \Delta_1$. Therefore, $\langle \sigma, \rho \rangle \geq -1$ for all points $\rho \in \Delta$ and $\langle \sigma, \rho \rangle = -1$ if and only if $\rho = \eta$, so, by definition, $\sigma$ is in $\nabla_1$. Since, by \cite[Proposition 6.3]{cycitv}, all points of $\nabla \cap N$ are either in $\nabla_1$ or $\nabla_2$, it follows that no point of $(\Delta^\circ \setminus \nabla) \cap N$ is in the interior of a facet of $\Delta^\circ$. Thus for any point $\sigma \in (\Delta^\circ \setminus \nabla) \cap N$, the intersection $D_{\sigma} \cap W$ is nonempty.

It just remains to treat the case where $\sigma$ lies in a codimension $2$ face of $\Delta^{\circ}$. Since $\sigma$ lies in $(\Delta^\circ \setminus \nabla)\cap N$, by definition there must be some $\rho_1 \in \Delta_1$ and $\rho_2 \in \Delta_2$ so that $\langle \rho_1 ,\sigma \rangle = \langle \rho_2 ,\sigma \rangle = -1$. Therefore, $\Gamma(\sigma)^\vee$ contains points in both $\Delta_1$ and $\Delta_2$, so is a face of neither. Given this, \cite[Proposition 6.3]{cycitv} implies that $\Gamma(\sigma)^\vee$ does not contain any points in its relative interior. So $\ell^*(\Gamma(\sigma)^\vee) = 0$ and hence $D_\sigma \cap W$ has a single irreducible component.
\end{proof}

\begin{remark}
For $W$ of arbitrary dimension, the same proof works to find the number of components of the member of the pencil $\calP$ corresponding to $[s:t]=[0:1]$. However, if $\dim(W) \geq 4$ we do not know whether this may be interpreted as a count of components of a singular fibre in a fibration on some birational model $\hat{W}$ of $W$ (see Remark \ref{rem:higherdim}).
\end{remark}

  Next we show that this number also has meaning with respect to the mirror Calabi-Yau variety $V$ and its degeneration to the union of $X_1$ and $X_2$. 

\begin{proposition}\label{prop:Cgenus}
If  $\dim(\nabla_1) = \dim(\nabla_2)= 4$, then $V \cap X_1 \cap X_2$ is an irreducible curve $C$ of genus 
\[g(C) = \# (\Delta^\circ \setminus \nabla) \cap N - 1.\]
\end{proposition}
\begin{proof}
By construction, $C$ is a complete intersection of sections of the line bundles $\omega_{X_\Delta}^{-1}$ (which determines $V$), $\mathcal{L}_1 = \mathcal{O}_{X_\Delta}(\sum_{\rho \in \Delta_1 \cap N}D_\rho)$ (which determines $X_1$), and $\mathcal{L}_2 = \mathcal{O}_{X_\Delta}(\sum_{\rho \in \Delta_2 \cap N}D_\rho)$ (which determines $X_2$). The Koszul complex resolving $\mathcal{O}_C$ is thus given by
\[\omega_{X_\Delta}^2 \rightarrow (\mathcal{L}^{-1}_1\otimes \omega_{X_\Delta}) \oplus (\mathcal{L}^{-1}_2\otimes \omega_{X_\Delta}) \oplus \omega_{X_\Delta} \rightarrow \mathcal{L}^{-1}_1 \oplus \mathcal{L}^{-1}_2 \oplus \omega_{X_\Delta} \rightarrow \mathcal{O}_{X_\Delta}\]

The corresponding second spectral sequence converges to $ H^i(C,\mathcal{O}_C[3])$, so
\[\bigoplus_{p+q = i+3} {''E}_\infty^{p,q}  \cong  H^i(C,\mathcal{O}_C).\]
The relevant portion of $''E_1^{p,q}$ is given by
\begin{equation*}
\begin{array}{cccccl}
 H^4(\omega_{X_\Delta}^2 ) & \rightarrow &  H^4(\mathcal{L}^{-1}_1\otimes \omega_{X_\Delta}) \oplus  H^4(\mathcal{L}^{-1}_2\otimes \omega_{X_\Delta}) \oplus  H^4(\omega_{X_\Delta}) & \rightarrow &\mathbb{C} &\rightarrow\ 0\\
0& \rightarrow& 0 &\rightarrow &0 &\rightarrow\ 0\\
0& \rightarrow& 0 &\rightarrow &0 &\rightarrow\ 0\\
0& \rightarrow& 0 &\rightarrow &0 &\rightarrow\ 0\\
0& \rightarrow& 0 &\rightarrow &0 &\rightarrow\ \mathbb{C}
\end{array}
\end{equation*}

Now, by \cite[Theorem 2.5]{cycitv}, we know that 
\begin{align*}
h^4(\omega_{X_\Delta}^2) &= {\ell^*(2\Delta^\circ)},\\ 
h^4(\omega_{X_\Delta}) &= 1 \\ 
h^4(\mathcal{L}^{-1}_1\otimes \omega_{X_\Delta}) &= \ell^*(\nabla_1 + \Delta^\circ), \\  h^4(\mathcal{L}^{-1}_2\otimes \omega_{X_\Delta}) &= \ell^*(\nabla_2 + \Delta^\circ).
\end{align*}
It is not then hard to see that this spectral sequence degenerates at the $''E_2$ term and $h^0(\mathcal{O}_C) = 1$, hence $C$ is irreducible. Since $h^i(\mathcal{O}_C) = 0$ for $i > 1$, we have that the top row of $''E_1^{p,q}$ above is exact except at the left-most term. Thus we can compute that 
\[g(C) = \ell^*(2\Delta^\circ) - (\ell^*(\nabla_1 + \Delta^\circ) + \ell^*(\nabla_2 + \Delta^\circ)).\]

It remains to show that this is precisely the number of points in $(\Delta^\circ \setminus \nabla)\cap N$. For this we need a small lemma.

\begin{lemma}
If  $Q$ is either $\nabla_i$ or $\Delta^\circ$, the number $\ell^*(Q + \Delta^\circ)$ is equal to $\ell(Q)$, where $\ell(Q)$ denotes the number of lattice points in $Q$.
\end{lemma}

\begin{proof}
The polytope $\Delta^\circ$ is defined by the inequalities $\langle \sigma,\rho \rangle \geq -1$ for all points $\rho \in \Delta$. Similarly, $\nabla_1$ is defined by the inequalities $\langle\sigma,\rho \rangle \geq -1$ for all points $\rho \in \Delta_1$ and $\langle \sigma,\rho \rangle \geq 0$ for all points $\rho \in \Delta_2$. We shall prove the lemma for $Q = \nabla_1$; the other cases are analogous. 

Now, the polytope $\nabla_1 + \Delta^\circ$ is defined by the inequalities $\langle \sigma, \rho \rangle \geq -2$ for $\rho \in \Delta_1$ and $\langle \sigma, \rho \rangle \geq -1$ for $\rho \in \Delta_2$. Therefore a point in the interior of $\nabla_1 + \Delta^\circ$ satisfies these inequalities strictly, and thus any lattice point  in the interior of $(\nabla_1 + \Delta^\circ)$ has $\langle \sigma, \rho \rangle \geq 0$ for all $\rho \in \Delta_2$ and $\langle \sigma, \rho \rangle \geq -1$ for all $\rho \in \Delta_1$. But this is just the set of all lattice points in $\nabla_1$.
\end{proof}
 
From this lemma, we see that 
\[\ell^*(2\Delta^\circ) - (\ell^*(\nabla_1 + \Delta^\circ) + \ell^*(\nabla_2 + \Delta^\circ)) = \ell(\Delta^\circ) - \ell(\nabla_1) - \ell(\nabla_1).\]
Moreover, \cite[Proposition 6.3]{cycitv} shows that all lattice points of $\nabla$ are lattice points of either $\nabla_1$ or $\nabla_2$, so this is equal to
\[\ell(\Delta^\circ) - \ell(\nabla) + 1 = \# (\Delta^\circ \setminus \nabla) \cap N+1;\]
here the extra $(+1)$ term corresponds to the fact that we have over-counted the origin, which is the intersection of $\nabla_1$ and $\nabla_2$. This completes the proof of Proposition \ref{prop:Cgenus}.
\end{proof}

\begin{remark}
A very minor modification of this proof shows that, in the case where $V$ has dimension $d \geq 3$, we have $h^{d-2,0}(V \cap X_1 \cap X_2) = \# (\Delta^\circ \setminus \nabla) \cap N -1$. If $\dim(V) = 2$, then $h^{0}(V \cap X_1 \cap X_2) = \# (\Delta^\circ \setminus \nabla) \cap N$.
\end{remark}

Putting everything together, we obtain the following theorem.

\begin{theorem}
If $\dim(V) = \dim(W) = d \geq 3$, then the member of the pencil of hypersurfaces $\mathcal{P}$ corresponding to $[s:t]=[0:1]$ has exactly $h^{d-2,0}(V \cap X_1\cap X_2) +1$ components. If $\dim(V) = \dim(W) = 2$, then $V \cap X_1\cap X_2$ is a set of points and the member of the pencil of hypersurfaces $\mathcal{P}$ corresponding to $[s:t]=[0:1]$ has exactly $\#(V \cap X_1\cap X_2)$ components.
\end{theorem}

Next we analyze the rest of the members of the pencil $\calP$ on $W$. Our goal is to show that the members corresponding to $[1:0]$ and $[1:1]$ are essentially the singular fibres of the LG models of $X_1$ and $X_2$. Thus there is a very real sense in which the pencil $\calP$ on ${W}$ is collecting information about the LG models of $X_1$ and $X_2$. 

First, however, we describe how these LG models are constructed. In \cite{harderthesis}, it is shown that the na\"ive compactification of Givental's \cite{mttci} Landau-Ginzburg model for a complete intersection $X$ in a toric variety $X_\Delta$ is smooth if $X$ has dimension less than or equal to $3$, and otherwise has only mild singularities.

This compactification is defined as follows. Assume that we have a polytope $\Delta$ and a nef partition $\Delta_1,\Delta_2$ of $\Delta$, so that $\Delta_1$ and $\Delta_2$ contain no interior points. In this setting, a general enough global section of the line bundle $\mathcal{L}_1$ associated to $\Delta_1$ determines a quasi-Fano hypersurface $X$ in $X_\Delta$. The compactified version of Givental's LG model for $X$ is then the complete intersection $Y \subset X_\nabla \times \mathbb{A}^1$ cut out by the equations
\begin{align*}\label{eq:penc}
\sum_{\rho \in \Delta_{1} \cap M} a_\rho \prod_{\sigma \in \nabla\cap N} z_\sigma^{\langle\sigma,\rho \rangle - \sigma^1_\mathrm{min}} &=0  \\
ta_0\prod_{\sigma \in \nabla_{2}\cap N\setminus 0_N}z_\sigma - \sum_{\rho \in \Delta_{2}\cap M \setminus 0_M} a_\rho \prod_{\sigma \in \partial\nabla\cap N} z_\sigma^{\langle\sigma,\rho \rangle - \sigma^{2}_\mathrm{min}}&=0 
\end{align*}
where $a_\rho$ are complex constants, $t$ is the coordinate on $\mathbb{A}^1$, and $\sigma^i_\mathrm{min}$ is $-1$ if $\sigma$ is in $\nabla_i$ and $0$ otherwise. The superpotential $\mathsf{w}$ is just projection of this complete intersection onto $\mathbb{A}^1$. It is shown in \cite{harderthesis} that $(Y,\mathsf{w})$ has the expected properties for an LG model of $X$.

With this in place, we find:

\begin{theorem}\label{thm:fib-bir}
The members of the pencil $\calP$ corresponding to $[s:t] = [1:1]$ and $[1:0]$ are birational to the fibres over $0$ of the LG models $(Y_1,\mathsf{w}_1)$ and $(Y_2,\mathsf{w}_2)$ of $X_1$ and $X_2$ respectively. In fact, for any choice of ${W}$ and hypersurface $\mathcal{P}([1:t])$ with $t\in \CC$, there is a choice of LG model $(Y,\mathsf{w})$ of either $X_1$ or $X_2$ so that $\mathcal{P}([1:t])$ is birational to a fibre of $(Y,\mathsf{w})$.
\end{theorem}

\begin{proof}
Recall from Section \ref{subsec:mirrorpencils} that we have an expression for a birational model of $W$ as the complete intersection in $\mathbb{P}^1[s,t] \times X_\Delta^{\circ}$ given by the vanishing of
\begin{align*}
f_1 & := s \sum_{\rho \in \Delta_1 \cap M\setminus 0_M} a_\rho \prod_{\sigma \in \partial \Delta^\circ \cap N} z_\sigma^{\langle \sigma, \rho\rangle +1} - ta_0 \prod_{\sigma \in \partial\Delta^\circ \cap N} z_\sigma \\ f_2 &:= s \sum_{\rho \in \Delta_2 \cap M\setminus 0_M} a_\rho \prod_{\sigma \in \partial \Delta^\circ \cap N} z_\sigma^{\langle \sigma, \rho\rangle +1} - (s-t)a_0\prod_{\sigma \in \partial\Delta^\circ \cap N} z_\sigma.
\end{align*}
Note that $f_1$ has a factor of $\prod_{\sigma \in \nabla_2 \cap N \setminus 0_N} z_\rho$ by the definition of $\nabla_2$, and an analogous statement holds for $f_2$, by the definition of $\nabla_2$. 

Now if we let $[s:t]=[1:1]$, then we obtain the complete intersection of
\begin{align*}
f_1|_{[1:1]} & =\sum_{\rho \in \Delta_1 \cap M} a_\rho \prod_{\sigma \in \partial \Delta^\circ \cap N} z_\sigma^{\langle \sigma, \rho\rangle +1}  \\ 
f_2|_{[1:1]} &=  \sum_{\rho \in \Delta_2 \cap M\setminus 0_M} a_\rho \prod_{\sigma \in \partial \Delta^\circ \cap N} z_\sigma^{\langle \sigma, \rho\rangle +1}.
\end{align*}
Note that this is precisely the complete intersection determining the fibre over $0$ of the LG model of $X_1$, except compactified in $X_{\Delta^\circ}$ instead of $X_\nabla$. 

To compare these compactifications, define a birational map $\varphi$ from $X_{\Delta^\circ}$ to $X_\nabla$, which sends $z_\sigma$ to $z_\sigma$ if $\sigma \in \partial \nabla \cap N$. The restriction of $\varphi$ to the complement of the torus invariant loci of codimension $\geq 2$ simply has the effect of removing those codimension $1$ tori corresponding to points $\sigma \in (\Delta^\circ \setminus \nabla) \cap N$. Thus we see that $\varphi$ induces a birational map between $\calP([1:1])$ and the fibre over $0$ of $(Y_1,\mathsf{w}_1)$ if no components of $\mathcal{P}([1:1])$ are contained in torus invariant loci of $X_{\Delta^\circ}$ of codimension $\geq 2$, and no component of $\mathcal{P}([1:1])$ is contained in a divisor $D_\sigma$ for $\sigma \in (\Delta^\circ\setminus \nabla) \cap N$. 

The first of these two claims is trivial: since each component of $\mathcal{P}([1:1])$ is of codimension $2$ in $X_{\Delta^\circ}$, it is contained in a codimension $\geq 2$ torus invariant subvariety of $X_{\Delta^\circ}$ if and only if it is the closure of such a torus invariant subvariety. Since $W$ contains no torus invariant subvarieties of $X_{\Delta^\circ}$, this cannot happen. The second claim follows from the fact that if $\sigma \in (\Delta^\circ \setminus \nabla) \cap N$, then $D_\sigma \cap W$ is in $\mathcal{P}([0:1])$. Thus $\mathcal{P}([1:1])$ can only intersect $D_\sigma$ in at most a codimension $2$ subvariety of $X_{\Delta^\circ}$. 

An identical argument suffices to show that $\mathcal{P}([1:0])$ is birational to the fibre over $0$ of $(Y_2,\mathsf{w}_2)$, and in fact this shows that for any $p \in \mathbb{P}^1\setminus \{[0:1]\}$, for a generic choice of $W$ there is  a choice of $i \in \{1,2\}$ and an LG model $(Y_i,\mathsf{w}_i)$ so that $\mathcal{P}(p)$ is birational to a fibre of $(Y_i,\mathsf{w}_i)$.
\end{proof}

As a philosophical remark, this proves that all of the interesting data surrounding the pencil $\calP$ on ${W}$ is related to either the LG models of $X_1$ and $X_2$, or the variety $V \cap X_1 \cap X_2$. Indeed, this is the same data as was used to determine the Tyurin degeneration of $V$: $X_1$ and $X_2$ were the quasi-Fano hypersurfaces in $X_{\Delta}$, and $V \cap X_1 \cap X_2$ was the locus that needed to be blown up in $X_1$ to obtain a smooth degeneration. 

In the case where both $X_1$ and $X_2$ are pullbacks to $X_\Delta$ of ample hypersurfaces in $\mathbb{P}_\Delta$ along the mpcp resolution map, then we can say even more using the following theorem, proved in \cite{harderthesis}.

\begin{theorem}\cite{harderthesis}\label{thm:andrewsthm} With notation as above, suppose in addition that $X_i$ is the pull-back to $X_{\Delta}$ of an ample hypersurface in $\PP_{\Delta}$ and that $\dim{\Delta} = d \geq 4$. Let $(Y_i,\mathsf{w}_i)$ denote the LG model of $X_i$ and let $\rho_0$ denote the number of components in its singular fibre over $0$. Then
\[\rho_0 = h^{d-2,1}(X_i)+1.\]
\end{theorem}

From this, we immediately obtain:

\begin{corollary}\label{cor:3foldmirror}
Suppose that $\dim(V) = \dim(W) = 3$ and let $\rho_p$ be the number of irreducible components in the member of the pencil $\calP$ corresponding to $p \in \mathbb{P}^1$. If $X_1$ and $X_2$ are pullbacks of ample hypersurfaces in $\mathbb{P}_\Delta$, then
\begin{itemize}
\item $\rho_{[1:0]} = h^{2,1}(X_1) + 1$,
\item $\rho_{[1:1]} = h^{2,1}(X_2) +1$,
\item $\rho_{[0:1]} = h^{1,0}(C) +1$.
\end{itemize}
\end{corollary}

\begin{remark} Note that the same result is true for $\hat{W}$, as the birational transformation from $W$ to $\hat{W}$ is an isomorphism in codimension $1$. Thus the preceding corollary can be interpreted as a count of components in singular fibres of the K3 surface fibration $\pi\colon \hat{W} \to \PP^1$ on $\hat{W}$.
\end{remark}

Now let $\ell$ be the rank of the image of the restriction map $ H^2(\hat{W},\mathbb{C}) \rightarrow  H^2(S,\mathbb{C})$, for $S$ a smooth fibre of $\pi$. Using \cite[Lemma 3.2]{cytfmqk3s}, one can easily show that 
\[h^{1,1}(\hat{W}) = \sum_{p \in \mathbb{P}^1} (\rho_p - 1) + \ell + 1\]
Moreover, by \cite{dpmscyhtv} we see that $h^{2,1}(V) = h^{1,1}({W}) = h^{1,1}(\hat{W})$. So, noting that $h^{2,1}(\hat{X}_1) = h^{2,1}(X_1) + g(C)$ (see, for instance, \cite[Theorem 7.31]{htcagi}), Theorem \ref{thm:lee} gives
\[\sum_{p \in \mathbb{P}^1\setminus \{[1:0],[0:1],[1:1]\}}(\rho_p-1) + \ell + k = 20.\]

This implies that if Dolgachev-Nikulin mirror symmetry does not hold (in a precise sense) for the K3 surfaces associated to the nef partition $\Delta_1,\Delta_2$ and their Batyrev-Borisov duals, then this failure is seen by the fibres of the fibration $\pi\colon \hat{W} \to \PP^1$ away from the points in the set $\{[1:0],[0:1],[1:1]\}$.
\medskip

Finally, one may ask whether an analogue of Corollary \ref{cor:3foldmirror} holds when $V$ and $W$ are K3 surfaces. The difficulty here is in proving an analogue of Theorem \ref{thm:andrewsthm}: for subtle combinatorial reasons, the proof given in \cite{harderthesis} does not easily generalize to the K3 surface case. However, we expect the following conjecture to hold in this case:

\begin{conjecture}\label{conj:K3mirror} Suppose that $\dim(V) = \dim(W) = 2$ and let $\rho_p$ be the number of irreducible components in the fibre of the elliptic fibration $\pi\colon W \to \PP^1$ over $p \in \mathbb{P}^1$. If $X_1$ and $X_2$ are pullbacks of ample hypersurfaces in $\mathbb{P}_\Delta$, then
\begin{itemize}
\item $\rho_{[1:0]} = h^{1,1}(X_1) - h^{1,1}(X_{\Delta}) + 1$,
\item $\rho_{[1:1]} = h^{1,1}(X_2) - h^{1,1}(X_{\Delta}) + 1$,
\item $\rho_{[0:1]}$ is the number of points in $V \cap X_1 \cap X_2$ and the corresponding singular fibre is semistable (Kodaira type $I_n$), and
\item all other fibres of $\pi$ are irreducible.
\end{itemize}
\end{conjecture}

We will illustrate this conjecture with an example.

\begin{example}[Anticanonical hypersurfaces in $(\mathbb{P}^1)^3$]\label{ex:P1P1P1}
Let us take $V$ to be an anticanonical hypersurface in $(\mathbb{P}^1)^3$. This is a K3 surface with Picard lattice of rank $3$, isomorphic to the lattice with Gram matrix
\[\left( \begin{matrix} 0 & 2 & 2 \\ 2 & 0 & 2 \\ 2 & 2& 0 \end{matrix} \right).\]

There is a Tyurin degeneration of $V$ to a union of two $(1,1,1)$ divisors $X_1,X_2$ in $(\mathbb{P}^1)^3$. The intersection $V \cap X_1 \cap X_2$ is $12$ points. On the other side, we see that there is an elliptic fibration on the mirror dual K3 surface $W$, which has three reducible fibres of types ${I}_{12}$, ${I}_2$ and $I_2$. 

The polytope $\Delta$ determining $(\mathbb{P}^1)^3$ has vertices $\sigma_0,\dots, \sigma_5$ given by the columns of the matrix
\[\left( \begin{matrix} 1 & -1 & 0 & 0 & 0 & 0 \\ 0 & 0 & 1 & -1 & 0 & 0 \\ 0 & 0 & 0 & 0& 1 & -1 \end{matrix} \right)\]
The appropriate nef partition is $E_1 = \{ \sigma_0,\sigma_2,\sigma_4\}$ and $E_2 = \{\sigma_1,\sigma_3,\sigma_5\}$, which has dual nef partition with 
\begin{align*}
\nabla_1 &= \mathrm{Conv}(\{ (1,0,0), (0,1,0),(0,0,1),(1,1,1),(1,1,0),(0,1,1),(0,1,1), (0,0,0)\}) \\
\nabla_2 &=- \nabla_1.
\end{align*}

We draw the polytopes $\Delta^\circ$, $\nabla_1$ and $\nabla_2$ in Figure \ref{fig:P1P1P1}. The leftmost picture in Figure \ref{fig:P1P1P1} is just the polytope $\Delta^\circ$, the middle picture denotes $\nabla_1$ and $\nabla_2$ and the picture on the right shows $(\Delta^\circ \setminus \nabla )\cap N$. It is clear from the description of the fibre over $[0:1]$ that it is actually semi-stable, so it follows from Kodaira's classification of singular fibres of elliptic fibrations that the resulting fibre is necessarily of type $\mathrm{I}_{12}$. The same cannot be said for the fibres over $[1:0]$ and $[1:1]$, which have two components each, since it is not necessarily true that these fibres have normal crossings. Kodaira's classification can only be used to determine that these fibres are either of type $\mathrm{I}_2$ or of type $\mathrm{III}$.
\end{example}

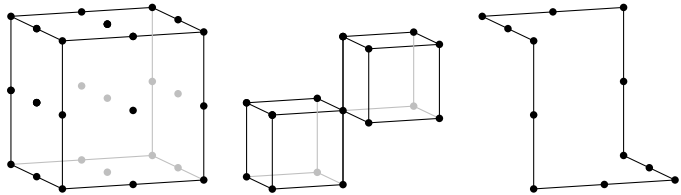
\begin{figure}
\centering{
\tdplotsetmaincoords{100}{20}
\begin{tikzpicture}[tdplot_main_coords]
	\draw[,-](1,1,1) -- (1,1,-1) ;	
	\draw[,-](1,1,1) -- (-1,1,1) ;
	\draw[,-](1,1,1) -- (1,-1,1) ;
	\draw[color=lightgray,-](1,1,-1) -- (1,-1,-1) ;
	\draw[,-](1,1,-1) -- (-1,1,-1) ;
	\draw[,-](-1,1,1) -- (-1,-1,1) ;
	\draw[,-](-1,-1,1) -- (-1,-1,-1) ;
	\draw[color=lightgray,-](-1,-1,-1) -- (1,-1,-1) ;
	\draw[,-](-1,-1,-1) -- (-1,1,-1) ;
	\draw[color=lightgray,-](1,-1,-1) -- (1,-1,1) ;
	\draw[,-](-1,1,-1) -- (-1,1,1) ;
	\draw[,-](1,-1,1) -- (-1,-1,1) ;

	\draw(-1,-1,-1) node[circle,fill,inner sep=1pt,color=black](){};
\draw(-1,0,-1) node[circle,fill,inner sep=1pt,color=black](){};
\draw(-1,-1,0) node[circle,fill,inner sep=1pt,color=black](){};
\draw(0,-1,-1) node[circle,fill,inner sep=1pt,color=lightgray](){};
\draw(-1,0,0) node[circle,fill,inner sep=1pt,color=black](){};
\draw(0,0,0) node[circle,fill,inner sep=1pt,color=lightgray](){};

\draw(1,-1,-1) node[circle,fill,inner sep=1pt,color=lightgray](){};
\draw(1,0,-1) node[circle,fill,inner sep=1pt,color=lightgray](){};

\draw(1,-1,0) node[circle,fill,inner sep=1pt,color=lightgray](){};
\draw(1,0,0) node[circle,fill,inner sep=1pt,color=lightgray](){};

\draw(-1,1,-1) node[circle,fill,inner sep=1pt,color=black](){};
\draw(-1,0,0) node[circle,fill,inner sep=1pt,color=black](){};
\draw(0,0,-1) node[circle,fill,inner sep=1pt,color=lightgray](){};
\draw(0,1,-1) node[circle,fill,inner sep=1pt,color=black](){};

\draw(0,1,0) node[circle,fill,inner sep=1pt,color=black](){};

	\draw(-1,-1,0) node[circle,fill,inner sep=1pt,color=black](){};
\draw(-1,0,0) node[circle,fill,inner sep=1pt,color=black](){};
\draw(-1,-1,1) node[circle,fill,inner sep=1pt,color=black](){};
\draw(0,-1,0) node[circle,fill,inner sep=1pt,color=lightgray](){};
\draw(-1,0,1) node[circle,fill,inner sep=1pt,color=black](){};
\draw(0,-1,1) node[circle,fill,inner sep=1pt,color=black](){};
\draw(0,0,1) node[circle,fill,inner sep=1pt,color=black](){};

\draw(-1,1,0) node[circle,fill,inner sep=1pt,color=black](){};
\draw(-1,0,1) node[circle,fill,inner sep=1pt,color=black](){};

\draw(-1,1,1) node[circle,fill,inner sep=1pt,color=black](){};
\draw(0,0,1) node[circle,fill,inner sep=1pt,color=black](){};
\draw(0,1,1) node[circle,fill,inner sep=1pt,color=black](){};

\draw(0,0,1) node[circle,fill,inner sep=1pt,color=black](){};
\draw(1,1,0) node[circle,fill,inner sep=1pt,color=black](){};
\draw(0,1,1) node[circle,fill,inner sep=1pt,color=black](){};
\draw(1,0,1) node[circle,fill,inner sep=1pt,color=black](){};
\draw(1,1,1) node[circle,fill,inner sep=1pt,color=black](){};
\draw(1,1,-1) node[circle,fill,inner sep=1pt,color=black](){};
\draw(1,-1,1) node[circle,fill,inner sep=1pt,color=black](){};
\end{tikzpicture}
\tdplotsetmaincoords{100}{20}
\begin{tikzpicture}[tdplot_main_coords]
	\draw[ -](1,1,1) -- (1,1,0) ;	
	\draw[-](1,1,1) -- (0,1,1) ;
	\draw[-](1,1,1) -- (1,0,1) ;

	\draw[,-](-1,-1,0) -- (-1,-1,-1) ;
	\draw[color=lightgray,-](-1,-1,-1) -- (0,-1,-1) ;
	\draw[,-](-1,-1,-1) -- (-1,0,-1) ;
	\draw[,-](-1,0,-1) -- (-1,0,0) ;
		\draw[color=lightgray,-](0,0,0) -- (1,0,0) ;
		\draw[,-](-1,0,0) -- (0,0,0) ;
	\draw[,-](0,1,0) -- (0,-1,0) ;
	\draw[,-](0,0,1) -- (0,0,-1) ;
	\draw[,-](0,0,1) -- (0,0,-1) ;
	\draw[,-](0,1,1) -- (0,1,0) ;
	\draw[,-](-1,-1,0) -- (-1,0,0) ;
	\draw[,-](-1,-1,0) -- (0,-1,0) ;
	\draw[color=lightgray,-](1,1,0) -- (1,0,0) ;
	\draw[,-](1,1,0) -- (0,1,0) ;

	\draw[color=lightgray,-](0,0,-1) -- (0,-1,-1) ;
	\draw[,-](0,0,-1) -- (-1,0,-1) ;
		\draw[-](0,0,1) -- (0,1,1) ;
	\draw[,-](0,0,1) -- (1,0,1) ;
	\draw[color=lightgray,-](0,-1,0) -- (0,-1,-1) ;
	\draw[color=lightgray,-](1,0,1) -- (1,0,0) ;

	\draw(-1,-1,-1) node[circle,fill,inner sep=1pt,color=black](){};
\draw(-1,0,-1) node[circle,fill,inner sep=1pt,color=black](){};
\draw(-1,-1,0) node[circle,fill,inner sep=1pt,color=black](){};
\draw(0,-1,-1) node[circle,fill,inner sep=1pt,color=lightgray](){};
\draw(-1,0,0) node[circle,fill,inner sep=1pt,color=black](){};
\draw(0,0,0) node[circle,fill,inner sep=1pt](){};

\draw(1,0,0) node[circle,fill,inner sep=1pt,color=lightgray](){};

\draw(-1,0,0) node[circle,fill,inner sep=1pt,color=black](){};
\draw(0,0,-1) node[circle,fill,inner sep=1pt,color=black](){};

\draw(0,1,0) node[circle,fill,inner sep=1pt,color=black](){};

	\draw(-1,-1,0) node[circle,fill,inner sep=1pt,color=black](){};
\draw(-1,0,0) node[circle,fill,inner sep=1pt,color=black](){};

\draw(0,-1,0) node[circle,fill,inner sep=1pt,color=black](){};

\draw(0,0,1) node[circle,fill,inner sep=1pt,color=black](){};

\draw(0,0,1) node[circle,fill,inner sep=1pt,color=black](){};

\draw(0,0,1) node[circle,fill,inner sep=1pt,color=black](){};
\draw(1,1,0) node[circle,fill,inner sep=1pt,color=black](){};
\draw(0,1,1) node[circle,fill,inner sep=1pt,color=black](){};
\draw(1,0,1) node[circle,fill,inner sep=1pt,color=black](){};
\draw(1,1,1) node[circle,fill,inner sep=1pt,color=black](){};
\end{tikzpicture}
\tdplotsetmaincoords{100}{20}
\begin{tikzpicture}[tdplot_main_coords]

	\draw[color=black,-](1,1,-1) -- (1,-1,-1) ;
	\draw[color=black,-](1,1,-1) -- (-1,1,-1) ;
	\draw[color=black,-](1,-1,-1) -- (1,-1,1) ;
	\draw[color=black,-](1,-1,1) -- (-1,-1,1) ;
	\draw[color=black,-](-1,1,1) -- (-1,-1,1) ;
	\draw[color=black,-](-1,1,1) -- (-1,1,-1) ;

\draw(1,-1,1) node[circle,fill,inner sep=1pt,color=black](){};

\draw(1,-1,-1) node[circle,fill,inner sep=1pt,color=black](){};
\draw(1,0,-1) node[circle,fill,inner sep=1pt,color=black](){};

\draw(1,-1,0) node[circle,fill,inner sep=1pt,color=black](){};

\draw(1,1,-1) node[circle,fill,inner sep=1pt,color=black](){};
\draw(-1,1,-1) node[circle,fill,inner sep=1pt,color=black](){};
\draw(0,1,-1) node[circle,fill,inner sep=1pt,color=black](){};

\draw(-1,-1,1) node[circle,fill,inner sep=1pt,color=black](){};
\draw(0,-1,1) node[circle,fill,inner sep=1pt,color=black](){};

\draw(-1,1,0) node[circle,fill,inner sep=1pt,color=black](){};
\draw(-1,0,1) node[circle,fill,inner sep=1pt,color=black](){};

\draw(-1,1,1) node[circle,fill,inner sep=1pt,color=black](){};
\end{tikzpicture}
\caption{Polytopes related to Example \ref{ex:P1P1P1}\label{fig:P1P1P1}}
}
\end{figure}

\section{Dolgachev-Nikulin mirror symmetry} \label{sec:DNforK3s}

Next we'll turn our attention to K3 surfaces. As noted in the previous section, in the setting of Batyrev-Borisov mirror symmetry, a refinement of the nef partition defining our K3 surface $V$ gives rise to both a Tyurin degeneration of $V$ and an elliptic fibration on the mirror K3 surface $W$. In the K3 surface case this appears to be part of a wider correspondence, which seems to have first been noticed by Dolgachev \cite{mslpk3s}, between Type II degenerations (of which Tyurin degenerations are an example) and elliptic fibrations on the Dolgachev-Nikulin mirror.

We begin by giving a precise statement of this correspondence. Suppose we have a (pseudo-ample) $L$-polarized K3 surface $V$, for some lattice $L$. To define the Dolgachev-Nikulin mirror of $V$, we first fix a primitive isotropic vector $f$ in the orthogonal complement $L^{\perp}$ of $L$ in the K3 lattice $\Lambda_{\mathrm{K3}} \cong H^{\oplus 3} \oplus E_8^{\oplus 2}$ (where $H$ denotes the hyperbolic plane lattice and $E_8$ is the negative definite root lattice $E_8$). With this in place, the Dolgachev-Nikulin mirror $W$ of $V$ is defined to be an $\check{L}$-polarized K3 surface, where 
\[\check{L} := (\ZZ f)^{\perp}_{L^{\perp}} / \ZZ f.\]
Note that this depends upon the choice of isotropic vector $f$.

By the discussion in \cite[Section 6]{mslpk3s} (see also \cite[Section 2.1]{cmsak3s}), fixing $f$ is equivalent to fixing a $0$-dimensional cusp (Type III point) in the Baily-Borel compactification $\overline{D}_L$ of the period domain $D_L$ of $L$-polarized K3 surfaces. Call this cusp $P$.

Then we have the following result, which is essentially contained in \cite[Remark 7.11]{mslpk3s}:

\begin{proposition}\label{prop:cuspfibration} With notation as above, there is a bijective correspondence between $1$-dimensional cusps in $\overline{D}_L$ that pass through $P$, and primitive isotropic vectors $e \in \check{L}$.
\end{proposition}

\begin{remark} Since $1$-dimensional cusps in $\overline{D}_L$ correspond to Type II degenerations of $V$ and isotropic vectors $e \in \check{L}$ correspond to elliptic fibrations on $W$, this gives rise to a correspondence between Type II degenerations of $V$ and elliptic fibrations on $W$ (up to automorphism).
\end{remark}

\begin{proof}[Proof of Proposition \ref{prop:cuspfibration}] Suppose first that we have a $1$-dimensional cusp $C \subset \overline{D}_L$ that passes through $P$. By \cite[Section 2.1]{cmsak3s}, such cusps are in bijection with rank two primitive isotropic sublattices of $L^{\perp}$ that contain the vector $f$. So $C$ gives rise to a sublattice $E$ of $L^{\perp}$ and, since $E$ is a primitive sublattice of $(\ZZ f)^{\perp}_{L^{\perp}}$ and $E$ contains $f$, we see that $E \cap \check{L}$ is a primitive isotropic sublattice of $\check{L}$ of rank $1$. Let $e$ denote a generator of this sublattice; then $e$ is a primitive isotropic vector in $\check{L}$.

Conversely, suppose we have a primitive isotropic vector $e \in \check{L}$. By definition of $\check{L}$, we have $\langle e,f\rangle = 0$. So the lattice $E$ spanned by $e$ and $f$ is a rank two primitive isotropic sublattice of $L^{\perp}$. But this gives rise to a $1$-dimensional cusp passing through $P$.
\end{proof}

We will now illustrate this correspondence in some explicit examples, which will provide a glimpse of some possible deeper structure.

\subsection{$H$-polarized K3's}\label{sec:Hpol}

We begin by looking at the Type II degenerations of $H$-polarized K3's. An $H$-polarized K3 surface may be constructed as an anticanonical hypersurface in an mpcp resolution of the weighted projective space $\WP(1,1,4,6)$. However, the defining polytope of this weighted projective space does not admit any nef partitions, so we cannot apply the theory of Section \ref{sec:BBduality} to study it.

Instead, we will try a different approach to comparing Type II degenerations and elliptic fibrations on the mirror. Note first that an $H$-polarized K3 surface $V$ naturally corresponds to a double cover of the Hirzebruch surface $\mathbb{F}_4$, ramified over a divisor in the linear system $|4s + 12f|$, where $s$ is the class of the $(-4)$-section and $f$ is the class of a fibre (when there is no risk of confusion, we will always denote the $(-n)$-section in a Hirzebruch surface $\FF_n$ by $s$ and a fibre of the ruling by $f$). So $H$-polarized K3 surfaces are equivalent to pairs $(\mathbb{F}_4,B)$, where $B \in |4s + 12f|$ is the branch divisor. 

The moduli space of such pairs admits a \emph{KSBA compactification} (see \cite{lcscmsp,msmgnws,hdasc,tadoss}), which has been studied in detail by Brunyate \cite{brunyatethesis}. This compactification admits a birational morphism to the usual Baily-Borel compactification of the moduli space of $H$-polarized K3 surfaces. The Type II degenerations occur along two boundary components in this KSBA compactification, which map to the two $1$-dimensional cusps in the Baily-Borel compactification. These two kinds of Type II degenerations may be described as follows:

\begin{enumerate}
\item Degenerate $V$ to a union $X_1^1 \cup_Z X_2^1$ defined as follows. $X_1^1$ is a double cover of $\mathbb{F}_4$ ramified over a smooth divisor in the linear system $|2s+12f|$, and $Z \subset X_1^1$ is an elliptic curve given by the pull-back of the $(-4)$-section. $X_2^1$ is also a double cover of $\FF_4$, this time ramified over a smooth divisor in the linear system $|s+4f|$, and $Z \subset X_2^1$ is the pull-back of a second smooth divisor in the linear system $|s+4f|$.

Riemann-Roch easily yields that there is a $17$-dimensional space of deformations of $(X_1^1,Z)$ preserving the double covering structure, and a $1$-dimensional space of deformations of $(X_2^1,Z)$. To glue these components together along $Z$, we need to ensure that the elliptic curves $Z$ in each component are isomorphic; this imposes a single gluing condition. The total configuration therefore has $17 + 1 -1 = 17$ moduli. It is easy to see that any variety $X_1^1 \cup_Z X_2^1$ defined in this way is $d$-semistable (in the sense of \cite{gsvnc}), so admits a smoothing to a K3 surface by \cite[Theorem 5.10]{gsvnc}. Thus, such degenerations lie along a $17$-dimensional boundary component (i.e. a boundary divisor) in the KSBA moduli space.

\item Degenerate $V$ to a union $X_1^2 \cup_Z X_2^2$, where both $X_1^2$ and $X_2^2$ are double covers of $\FF_2$ ramified over smooth divisors in the linear system $|4s+6f|$ and $Z$ is the pull-back of a fibre of the ruling on $\FF_2$. Riemann-Roch yields that each component $(X_i^2,Z)$ has $9$ deformations that preserve the double covering structure and, as before, there is a single gluing condition along $Z$. The total configuration therefore has $9+9-1 = 17$ moduli and is $d$-semistable, so gives another boundary divisor of the moduli space.
\end{enumerate}

Now we look at the mirror. Up to isometry, there are two $1$-dimensional cusps in $\overline{D}_H$, meeting in a unique $0$-dimensional cusp. So we only have one choice of mirror $W$: a K3 surface polarized by the lattice $M:= H \oplus E_8 \oplus E_8$. 

Now we match the Type II degenerations above with elliptic fibrations on the mirror. From \cite[Section 3.3]{milpk3s}, we know that an $M$-polarized K3 surface admits two elliptic fibrations:
\begin{enumerate}
\item The alternate fibration, which has an $I_{12}^*$ and six $I_1$'s. Note that the $I_{12}^*$ fibre has $17$ components, corresponding to the $17$ moduli of the component $X_1^1$ in degeneration (1) above.
\item The standard fibration, which has two $II^*$'s and four $I_1$'s. Note that each $II^*$ fibre has $9$ components, corresponding to the $9$ moduli of each of the components $X_i^2$ in the Type II degeneration (2) above.
\end{enumerate}

\subsection{K3 surfaces of degree two}\label{sec:deg2}

Now consider K3 surfaces of degree two (i.e. polarized by the rank one lattice $\langle 2 \rangle$). A K3 surface $V$ of degree $2$ may be constructed as an anticanonical hypersurface in the weighted projective space $\mathbb{WP}(1,1,1,3)$. This toric variety is determined by the polytope $\Delta \subset M_{\RR} \cong \RR^3$ with vertices given by the columns $\sigma_1,\sigma_2,\sigma_3,\sigma_4$ of the matrix
\[ \left( \begin{matrix} 1 & 0 & 0 & -1 \\ 0 & 1 & 0 & -1\\ 0 & 0 & 1 & -3 \end{matrix}\right).\]

Up to automorphism, there is only one bipartite nef partition of this polytope, given by $E_1 = \{\sigma_1,\sigma_2,\sigma_4\}$ and $E_2 = \{ \sigma_3 \}$. The dual nef partition in $N_\mathbb{R}$ has 
\begin{align*}
\nabla_1 &= \mathrm{Conv}(-e_1^* -e_2^*,-e_1^* -e_2^* + e_3^*, -e_1^* +2e_2^*, 2e_1^* -e_2^*)\\
\nabla_2 &= \mathrm{Conv}(0_N, -e_3^*, 3e_1^*-e_3^*, 3e_2^* - e_3^*)
\end{align*}
where $e_1,e_2$ and $e_3$ are basis vectors of $M_\mathbb{R}$ and $e_i^*$ are their duals.

As in Example \ref{ex:P1P1P1}, we may look at the polytopes $\nabla_1$ and $\nabla_2$ in comparison to $\Delta^\circ$ and note that the integral points in $\Delta^\circ \setminus \nabla$ form a cycle of length $18$. Thus the fibre over $[0:1]$ in the elliptic fibration on the Batyrev dual $W$ described in Section \ref{subsec:mirrorpencils} is of type $ I_{18}$. One can also check, using techniques described in \cite{harderthesis}, that the fibres over $[1:0]$ and $[1:1]$ are irreducible. 

Finally, in the degree two case it is well known that the Batyrev dual family of K3 surfaces is actually Dolgachev-Nikulin dual (see, for instance, \cite{lptk3s}), so a generic Batyrev dual K3 surface $W$ has Picard lattice $M_2 := H \oplus E_8 \oplus E_8 \oplus A_1$. The elliptic fibrations on such an $M_2$-polarized K3 surface were computed by Dolgachev \cite[Remark 7.11]{mslpk3s}; we see that one of them has an $I_{18}$ fibre and six $I_1$'s, as expected.
\medskip

However, it is known (see, for instance,  \cite[Section 6]{cmsak3s}) that the Baily-Borel compactification of the moduli space of K3 surfaces of degree two has four $1$-dimensional cusps, corresponding to four types of Type II degenerations, yet the example above only gives one. We can analyse the others using the same techniques that we used in the $H$-polarized case above.

Indeed, it is well known that a K3 surface of degree two $V$ naturally corresponds to a double cover of $\mathbb{P}^2$ ramified over a smooth sextic curve. The moduli space of K3 surfaces of degree two is therefore the same as the moduli space of pairs $(\mathbb{P}^2,B)$ where $B$ is a sextic curve. 

Alexeev and Thompson have studied a KSBA compactification \cite{mcmk3sd2} for the moduli space of such pairs. The Type II degenerations occur along four boundary components, which map to the four cusps in the Baily-Borel compactification. The four corresponding kinds of Type II degenerations may be described as follows:
\begin{enumerate}
\item Degenerate $V$ to a union $X_1^1 \cup X_2^1$, defined as follows. $X_1^1$ is a double cover of $\mathbb{F}_4$ ramified over a smooth divisor in the linear system $|2s+12f|$, and $Z \subset X_1^1$ is an elliptic curve given by the pull-back of the $(-4)$-section. $X_2^1$ is a double cover of $\PP^2$ ramified over a smooth conic, and $Z \subset X_2^1$ is the pull-back of a second smooth conic. This degeneration is in many ways analogous to degeneration (1) in the $H$-polarized case.

In this case the pair $(X_1^1,Z)$ admits $16$ deformations preserving the double covering structure, and $(X_2^1,Z)$ has two. As in the $H$-polarized case, there is a single gluing condition corresponding to choice of the elliptic curve $Z$, so the total configuration has  $17 + 2 - 1 = 18$ moduli. It is easy to see that all such varieties are $d$-semistable, so they give rise to an $18$-dimensional boundary component (i.e. a boundary divisor) in the KSBA moduli space.

\item Degenerate $V$ to a union $X_1^2 \cup_Z X_2^2$, defined as follows. $X_1^2$ is a double cover of $\FF_1$ ramified over a smooth divisor in the linear system $|2s+6f|$ and $Z \subset X_1^1$ is the pull-back of the $(-1)$-section. $X_2^2$ is a double cover of $\mathbb{P}^2$ ramified over a smooth quartic and $Z \subset X_2^2$ is the pull-back of a line.

Counting deformations that preserve the double cover structure, we see that $(X_1^2,Z)$ has $11$ deformations and the $(X_2^2,Z)$ component has $8$. However, as usual there is a gluing condition along $Z$, so the total configuration therefore has $11 + 8 - 1 = 18$ moduli. Since such varieties are all $d$-semistable, they correspond to a boundary divisor in the moduli space.

\item The last two cases are more interesting. In the first, we degenerate $V$ to a single component $X_1^3$, defined to be the double cover of a cubic cone ramified over a smooth anticanonical divisor. $X_1^3$ thus contains two log canonical singularities, corresponding to the pull-backs of the vertex of the cone. This double cover has $18$ moduli, and corresponds to a boundary divisor in the KSBA moduli space.

Interestingly, this degeneration may be seen to be equivalent to the degeneration given by the nef partition above. Indeed, if $X_1 \cup_Z X_2$ is the degeneration corresponding to this nef partition, we may obtain the degeneration $X_1^3$ by blowing up along $Z$, then contracting $X_1$ and $X_2$. These contractions give rise to the two log canonical singularities present in $X^3_1$.

\item The final case is most difficult. In \cite{mcmk3sd2}, the corresponding degeneration of $V$ is a union $X_1^4 \cup_Z X_2^4$, where each $X_i^4$ is a double cover of $\FF_2^0 \cong \WP(1,1,2)$ (obtained by taking a copy of $\FF_2$ and contracting the $(-2)$-section) ramified over a divisor in the linear system $\calO(6)$ and cyclically over the $A_1$ singularity, and $Z$ is the pull-back of a divisor in the linear system $\calO(1)$. This degeneration is in many ways analogous to degeneration (2) in the $H$-polarized case.

Counting moduli, we see that each component $(X_i^4,Z)$ has $9$ deformations that preserve the double covering structure, but there is, as always, one gluing condition along $Z$. We thus have $9+9-1 = 17$ deformations, which does not give the expected divisor on the boundary. However, there is also something wrong with this degeneration: one can easily check that it is not $d$-semistable, so does not appear as the central fibre in a smooth semistable degeneration of K3 surfaces.

There are two ways to resolve this. The first may be thought of as analogous to the blow-up of $V \cap X_1 \cap X_2$ in $X_1$ from Section \ref{sec:BBduality}: we simply blow up an arbitrary point on the image of $Z$ in one copy of $\mathbb{F}_2^0$. This corresponds to blowing up a pair of points on $Z$ in $X_1^4$, under the condition that the two points blown up are exchanged by the involution defining the double covering. After blowing up,  $X_1^4 \cup_Z X_2^4$ becomes $d$-semistable and acquires an extra modulus, corresponding to the choice of point to blow up, making the corresponding boundary component into a divisor.

However, there is also a second way to resolve this problem. Inspired by (3), above, we could try introducing a third component between $X_1^4$ and $X_2^4$; this has the appeal of maintaining the symmetry of the degenerate fibre. Such a central component $X_3^4$ can be constructed as follows. Take a copy of $\PP^2$ and blow up three points in general position, to get a $(-1)$-hexagon. Blow up one of the vertices of this hexagon again, then contract the two $(-2)$-curves that result. One obtains a surface that has two $A_1$ singularities, and the $(-1)$-hexagon becomes a pentagon whose sides have self-intersections $(-1,-1,-\frac{1}{2},0,-\frac{1}{2})$. The sum of the two $(-1)$-curves in this pentagon gives a ruling and $X_3^4$ is a double cover of this surface ramified over three of its fibres, as well as cyclically over both of the $A_1$ singularities. $X_3^4$ glues to $X_1^4$ and $X_2^4$ along two isomorphic elliptic curves $Z_1$ and $Z_2$, which are the pull-backs of the two $(-\frac{1}{2})$-curves. We thus obtain a degenerate fibre $X_1^4 \cup_{Z_1} X_3^4 \cup_{Z_2} X_2^4$.

Accounting for automorphisms, $(X_3^4,Z_1,Z_2)$ has $2$ deformations. So, with this component included, the total moduli count is $9 + 9 + 2 - 1 - 1 = 18$ (the two $(-1)$'s appear because there is a gluing condition associated to each double curve). Moreover, the fibre $X_1^4 \cup_{Z_1} X_3^4 \cup_{Z_2} X_2^4$ is $d$-semistable, so we get a boundary divisor in moduli.
\end{enumerate}

Now we again match with elliptic fibrations on the mirror. These are computed by Dolgachev \cite[Remark 7.11]{mslpk3s} as:
\begin{enumerate}
\item An elliptic fibration which has one $I_{12}^*$, one $I_2$, and four $I_1$'s. Note that the $I_{12}^*$ fibre has $17$ components, corresponding to the $17$ moduli of $(X_1^1,Z)$ in degeneration (1) above, and the $I_2$ fibre has two components, corresponding to the two moduli of $(X_2^1,Z)$.
\item An elliptic fibration with one fibre of type $I_{6}^*$, one fibre of type $III^*$, and three $I_1$'s. As above, the $I_{6}^*$ fibre has $11$ components, corresponding to the $11$ moduli of $(X_1^2,Z)$ in degeneration (4) above, and the $III^*$ fibre has eight components, corresponding to the eight moduli of $(X_2^2,Z)$.
\item An elliptic fibration with one fibre of type $I_{18}$ and six $I_1$'s. Once again, the $I_{18}$ fibre has $18$ components, corresponding to the $18$ moduli of the single component $X_3^1$ in degeneration (3), above.
\item An elliptic fibration which has two $II^*$'s, one $I_2$, and two $I_1$'s. Note that each $II^*$ fibre has $9$ components, corresponding to the $9$ moduli of $(X_1^4,Z)$ and $(X_2^4,Z)$ in degeneration (2) above, and the $I_2$ fibre has two components. These two components can be thought of as corresponding to the two points on $Z$ which are blown up to make $X_1 \cup_Z X_2$ $d$-semistable (c.f. Conjecture \ref{conj:K3mirror}), or as corresponding to the two moduli of the ``extra'' central component $(X_3^4,Z_1,Z_2)$. 
\end{enumerate}

\subsection{Discussion}\label{sec:discussion}

These results are highly suggestive of the idea that the correspondence between Tyurin degenerations and elliptic fibrations on Batyrev dual K3 surfaces, explored in Section \ref{sec:BBduality}, may extend to a broader correspondence between Type II degenerations of an $L$-polarized K3 surface $V$ and elliptic fibrations on its Dolgachev-Nikulin mirror $W$. Such a correspondence should have properties that generalize those described in Section \ref{sec:BBduality}; the aim of this section is to discuss the form that such properties may take. 

Firstly, cases (3) and (4) of the considerations in Section \ref{sec:deg2} suggest that there may be some correspondence between the choice of model for the Type II degeneration of $V$ and some properties of the elliptic fibration. Thinking about this in the context of the threefold philosophy outlined in Remark \ref{rem:threefoldphilosophy}, one might conjecture the following.

Suppose first that we have a Type II degeneration of $V$ to a configuration $X_1 \cup_Z X_2 \cup_Z \cdots \cup_Z X_k$, where each $X_i$ meets $X_{i-1}$ and $X_{i+1}$. Let $n_i$ be the number of deformations of $(X_i,Z)$ (where $Z$ denotes the double locus on $X_i$ and may have more than one component) that preserve some notion of lattice polarization -- this last condition is to ensure that the deformed Type II fibres still smooth to $L$-polarized K3 surfaces, and was arranged in the preceding examples by requiring that deformations preserve the double covering structure; it also accounts for the appearance of the $h^{1,1}(X_{\Delta})$ term in Conjecture \ref{conj:K3mirror}. Then the fibre dimension of the natural map from the moduli space of such pairs $(X_i,Z)$ to the moduli space of elliptic curves $Z$ is equal to $n_i -1$. 

Thus, noting that all elliptic double curves in the Type II degeneration are isomorphic, we see that such a Type II fibre should have $\sum_{i=1}^k (n_i-1) + 1$ deformations. We first conjecture that this number should equal $19 - \ell$, where $\ell = \rank(L)$. Then, since the moduli space of $V$ (as an $L$-polarized K3 surface) is $(20-\ell)$-dimensional, such Type II degenerations should lie along codimension $1$ loci in an appropriate compactification. We would thus get a decomposition of the $20-\ell$ deformations of $V$ into contributions $(n_i -1)$ from deformations of each $X_i$, a contribution $1$ from deformations of $Z$, and $1$ for the codimension in the moduli space.

Now we look at the mirror picture. As proved in Proposition \ref{prop:cuspfibration}, the Type II degeneration of $V$ given above should correspond to an elliptic fibration $\pi\colon W \to \PP^1$ on the mirror $W$. We suggest that the decomposition $X_1 \cup_Z X_2 \cup_Z \cdots \cup_Z X_k$ of $V$ corresponds to a ``slicing'' of the $\PP^1$ base of $\pi$, so that each $X_i$ corresponds to a slice $S_i$, as follows. $S_1$ is a disc, which is glued along its boundary to one of the boundaries of an annulus $S_2$. The other boundary of the annulus $S_2$ is then glued to one of the boundaries of an annulus $S_3$, and so on, until the remaining boundary of the annulus $S_{k-1}$ is glued to a disc $S_k$.

The singular fibres of $\pi$ should then be apportioned amoungst the slices as follows. If $\rho_p$ denotes the number of components in the fibre of $\pi$ over $p$, then the slicing should satisfy
\[n_i - 1 = \sum_{p \in S_i} (\rho_p -1).\]
This gives a decomposition of the Picard rank $(20-\ell)$ of $W$ into contributions $n_i -1$ from the singular fibres lying on the slice $S_i$, a contribution $1$ from the class of a section, and a contribution $1$ from the class of a fibre.

Note that this is completely compatible with the Tyurin degeneration picture presented in Section \ref{sect:setup} and the Batyrev-Borisov picture of Conjecture \ref{conj:K3mirror}. Indeed, in the Tyurin degeneration picture the slicing has two pieces $S_1$ and $S_2$, which are the LG models of the two components $X_1$ and $X_2$ of the degenerate fibre $X_1 \cup_Z X_2$. Furthermore, in the setting of Conjecture 3.20 we also have a slicing into two pieces $S_1$ and $S_2$, one of which contains the point $[1:0]$ and the other of which contains $[1:1]$. Which slice the point $[0:1]$ falls into depends upon the choice of blow-up of $V \cap X_1 \cap X_2$: if we blow up $V \cap X_1 \cap X_2$ in $X_i$ (for $i \in \{1,2\}$), then the corresponding $I_n$ fibre appears in the slice $S_i$.

Finally, we describe how this slicing picture works in the $H$-polarized and degree $2$ cases considered above. In the $H$-polarized case we have the following two possibilities, numbered compatibly with Section \ref{sec:Hpol}. The labelling of the slices is chosen so that the slice corresponding to the component $X_i^j$ is labelled $S_i^j$.
\begin{enumerate}
\item $\PP^1$ is sliced into two pieces $S_1^1$ and $S_2^1$, such that $S_1^1$ contains the $I_{12}^*$ fibre.
\item $\PP^1$ is sliced into two pieces $S_1^2$ and $S_2^2$, each of which contains a $II^*$ fibre.
\end{enumerate}

Moreover, in the degree two case we have the following four possibilities, numbered compatibly with Section \ref{sec:deg2}.
\begin{enumerate}
\item $\PP^1$ is sliced into two pieces $S_1^1$ and $S_2^1$, such that $S_1^1$ contains the $I_{12}^*$ fibre and $S_2^1$ contains the $I_2$.
\item $\PP^1$ is sliced into two pieces $S_1^2$ and $S_2^2$, such that $S_1^2$ contains the $I_{6}^*$ fibre and $S_2^2$ contains the $III^*$.
\item In this case we have two choices of slicing, corresponding to the two types of degeneration. For the degeneration given by the nef partition and Batyrev mirror symmetry, $\PP^1$ is sliced into two pieces $S_1$ and $S_2$, one of which contains the $I_{18}$ fibre. For the KSBA degeneration, we have a ``degenerate'' slicing of $\PP^1$ into a single piece, which contains all singular fibres.
\item In the final case we also have two choices of slicing, corresponding to the two types of degeneration. In the case where we blow up a pair of points, $\PP^1$ is sliced into two pieces $S_1^4$ and $S_2^4$, such that $S_1^4$ contains the $II^*$ fibre and the $I_2$ fibre, and $S_2^4$ contains the other $II^*$. In the case where we have three components in the degeneration, $\PP^1$ is sliced into three pieces $S_1^4$, $S_3^4$ and $S_2^4$, such that $S_1^4$ and $S_2^4$ are discs containing one $II^*$ fibre each, and $S_4^3$ is an annulus containing the $I_2$ fibre.
\end{enumerate}

\section{Beyond Batyrev-Borisov mirror symmetry for threefolds} \label{sec:beyond3folds}

The aim of this section is to provide some evidence that the ideas presented in Section \ref{sect:setup} also hold for threefolds outside the toric setting considered in Section \ref{sec:BBduality}. We begin by showing that classical mirror symmetry suggests a correspondence between Tyurin degenerations and K3 fibrations on mirror dual pairs of Calabi-Yau threefolds. This should be thought of as a threefold analogue of Proposition \ref{prop:cuspfibration}. Then we specialize our discussion to the case of threefolds fibred by mirror quartics, as studied in \cite{cytfmqk3s}, and show that, in that setting, the correspondence predicted by classical mirror symmetry is consistent with the construction presented in Section \ref{sect:setup}.

\subsection{Classical mirror symmetry for threefolds}
\label{sect:MScones}

Classical mirror symmetry predicts that if $V$ and $W$ are mirror dual Calabi-Yau threefolds, then there is a relation between monodromy operators acting on $H^3(V,\mathbb{Q})$ and divisors in the closure of the K\"ahler cone of $W$. We will briefly sketch some of the details of this relationship here, the interested reader may find more details in \cite[Chapter 6]{msag}.

Suppose that $\calV \rightarrow (\Delta^*)^n$ is a family of Calabi-Yau threefolds over the punctured polydisc, with fibre $\calV_t = V$ above some $t \in (\Delta^*)^n$. For each  $i \in \{1,\ldots,n\}$, let $T_i$ be the unipotent monodromy operator acting on $H^3(V,\mathbb{Q})$ coming from the loop $(t_1,\dots,t_{i-1}, e^{2\pi i t},t_{i+1},\dots,t_n)$, where $(t_1,\dots, t_{i-1},t_{i+1},\dots,t_n)$ are fixed constants, and let $N_i = \log(T_i)$. The family $\calV$ is said to have \emph{maximally unipotent monodromy} at $(0,\dots,0)$ if 
\begin{enumerate}
\item for any $n$-tuple $(a_1,\dots,a_n)$ of positive integers, the weight filtration $W_{\bullet}$ on $H^3(V,\mathbb{Q})$ induced by $\sum_{i=1}^n a_i N_i$ has $\dim W_0 = \dim W_1 =1$ and $\dim W_2 = n +1$, and
\item if $g_0,\dots, g_n$ is a basis of $W_2$ chosen so that $g_0$ spans $W_0$, and $m_{ij}$ are defined by $N_i g_j = m_{ij}g_0$, then the matrix $(m_{ij})$ is invertible.
\end{enumerate}

If $\calV$ has maximally unipotent monodromy, then mirror symmetry should produce a map which assigns to each $N_i$ a divisor $D_i$ in the closure of the K\"ahler cone of $W$. Moreover, there should be an identification under mirror symmetry which gives an isomorphism $H^{3-i,i}(V) \cong H^{i,i}(W)$, and hence an isomorphism $H^{3}(V,\mathbb{C}) \cong \bigoplus_{i=0}^3 H^{i,i}(W)$, so that the action of $N_i$ on $H^3(V,\mathbb{C})$ agrees with the action of the cup product operator $J_i(-) = (-)\cup [-D_i]$ under this correspondence. Thus, for any $n$-tuple $(a_1,\ldots,a_n)$ of non-negative integers, the weight filtration  on $H^3(V,\mathbb{C})$ induced by $N := \sum_{i=1}^n a_iN_i$ should be mirrored by the filtration on $\bigoplus_{i=0}^3 H^{i,i}(W)$ induced by $J:= \sum_{i=1}^n a_iJ_i$, and the limit Hodge decomposition should correspond to the decomposition $\bigoplus_{i=0}^3H^{i,i}(W)$.

Now we specialize this discussion to the case of a Tyurin degeneration of Calabi-Yau threefolds $\calV \rightarrow \Delta$. As in the previous sections, we write the central fibre of $\calV$ as $X_1 \cup_Z X_2$ and let $V$ denote a general fibre. Let $T$ be the monodromy operator acting on $H^3(V,\QQ)$ associated to a counterclockwise loop around $0$. In order to apply the predictions of mirror symmetry, we assume that $T$ may be identified with a loop $\prod_{i=1}^n T_i^{a_i}$ around a point of maximally unipotent monodromy in the complex moduli space of $V$, where $T_i$ are as above and $a_i$ are non-negative integers. Define $N := \log(T) = \sum_{i=1}^n a_iN_i$. We will use the Clemens-Schmid exact sequence associated to $N$ to compute the limit mixed Hodge structure on $H^3(V)$, then see what this allows us to deduce about the mirror threefold $W$. 

\begin{remark} We note that the Tyurin degeneration $\calV$ cannot have maximally unipotent monodromy, for purely topological reasons (see, for instance, \cite[Corollary 2]{csesa}), so $T$ must correspond to a loop around some positive-dimensional boundary component of the compactified complex moduli space of $V$. In particular, this implies that some of the $a_i$ must be zero.
\end{remark}

We begin by looking at the mixed Hodge structure on $H^3(\calV)$ given by Griffiths and Schmid \cite[Section 4]{rdhtdtr}. The weight filtration $W_{\bullet}$ on $H^3(\calV,\mathbb{Q})$ has
\[\Gr^W_3 = H^3(X_1,\mathbb{Q}) \oplus H^3(X_2,\mathbb{Q})\]
and, if $r_1$ and $r_2$ are the restriction maps $r_i\colon H^2(X_i,\mathbb{Q}) \rightarrow H^2(Z,\mathbb{Q})$, then
\[\Gr^W_{2} = H^2(Z,\mathbb{Q}) / (\mathrm{im}(r_1) + \mathrm{im}(r_2)).\]
These weight graded pieces are then equipped with the appropriate Hodge filtrations. Define integers $u := \rank( \Gr_2^W) -2$ and $v := \frac{1}{2} \rank (\Gr_3^W)$. Noting that $K_{X_i}$ is anti-effective, so that $h^{3,0}(X_i) = 0$, we see that $v = h^{2,1}(X_1) + h^{2,1}(X_2)$. 

The Clemens-Schmid exact sequence gives us an exact sequence of mixed Hodge structures 
\[\cdots \longrightarrow H_5(\calV) \longrightarrow H^3(\calV) \stackrel{i^*}{\longrightarrow} H^3_\mathrm{lim}(V) \stackrel{N}{\longrightarrow} H^3_\mathrm{lim}(V) \longrightarrow H_3(\calV) \longrightarrow \cdots \]
where $i^*$ is the pull-back on cohomology induced by the inclusion $i\colon V \hookrightarrow \calV$.

\begin{lemma} $H_5(\calV) = 0$, so the map $i^*$ is an injection.\end{lemma}
\begin{proof} The Mayer-Vietoris sequence for $X_1 \cup_ZX_2$ gives 
\[\cdots \longrightarrow H_5(X_1) \oplus H_5(X_2) \longrightarrow H_5(\calV) \longrightarrow H_4(Z) \stackrel{\alpha}{\longrightarrow} H_4(X_1) \oplus H_4(X_2) \longrightarrow \cdots,\]
where the map $\alpha$ is induced by the inclusions $Z \hookrightarrow X_i$.

Now, $H_5(X_1) \oplus H_5(X_2)$ vanishes by Poincar\'{e} duality and the assumption that $h^{0,1}(X_i) = 0$. Moreover, as $Z$ is an effective anticanonical divisor in both $X_1$ and $X_2$, the image of the class $[Z] \in H_4(Z)$ of $Z$ under $\alpha$ is non-trivial. But $[Z]$ generates $H_4(Z)$, so $\alpha$ must be injective. Thus the sequence above gives $H_5(\calV) = 0$.
\end{proof}

Applying this lemma and some standard results on the Clemens-Schmid exact sequence (see, for instance, \cite{csesa}) we obtain the following limit mixed Hodge structure on $H^3(V)$
\[\begin{array}{rcccc}
& \Gr_F^3 & \Gr_F^2 & \Gr_F^1 & \Gr_F^0 \\
\Gr^M_4 &\mathbb{C} & \mathbb{C}^u & \mathbb{C} & 0 \\
\Gr^M_3 &0 & \mathbb{C}^v & \mathbb{C}^v &0  \\
\Gr^M_2 & 0 & \mathbb{C} & \mathbb{C}^u & \CC
\end{array}\]
where $M_{\bullet}$ is the monodromy weight filtration induced by $N$ and $F^{\bullet}$ is the limit Hodge filtration. 

Therefore, the divisor $D = \sum_{i=1}^na_iD_i$ on $W$ which corresponds to $N$ under mirror symmetry should have 
\[\begin{array}{rcccc}
& H^{0,0}(Y^\vee) & H^{1,1}(W) & H^{2,2}(W) & H^{3,3}(W) \\
\mathrm{coimage}(J) & \CC & \mathbb{C}^u & \mathbb{C} & 0 \\
\ker(J)/ \mathrm{im}(J) &0 & \mathbb{C}^v & \mathbb{C}^v &0  \\
\mathrm{im}(J) & 0 & \mathbb{C} & \mathbb{C}^u & \CC
\end{array}\]
where, as before,  $J(-) = (-)\cup [-D]$ denotes the cup-product operator. In particular, we see that $J^2 = 0$. Since $D$ is in the closure of the K\"ahler cone of $W$, results of Oguiso \cite[Example 2.3]{ffsscy3f} show that $mD$ is the class of a fibre in a fibration of $W$ by K3 or abelian surfaces, for some positive integer $m$.

\begin{remark} Based on the ideas in the previous sections, we conjecture that $mD$ will always be the class of fibre in a K3 fibration on $W$. Oguiso \cite[Example 2.3]{ffsscy3f} gives a simple criterion to test for this: $mD$ defines a K3 fibration on $W$ if and only if $c_2(W) \cdot D >0$.
\end{remark}

In light of this remark, we will assume  throughout the remainder of this section that $mD$ defines a K3 fibration on $W$. Then the calculation above also shows that the classes in $\Pic(W)$ supported on fibres span a $v+1$ dimensional subspace, where one of these classes is $mD$ itself. Moreover, there is a rank $u$ subspace of $\Pic(W)$ with $J(\eta) \neq 0$ for each class $\eta\neq 0$ in this subspace. By the global invariant cycles theorem, classes in this second subspace come from monodromy invariant cycles on fibres of the K3-fibration on $W$. Thus the K3 surface fibration on $W$ induced by $mD$ is $\check{L}$-polarized (in the sense of \cite[Definition 2.1]{flpk3sm}), for some lattice $\check{L}$ of rank $u$.

Therefore we see that, if $V$ admits a Tyurin degeneration to a union of threefolds $X_1 \cup_Z X_2$, and if restriction of divisors from $X_1$ and $X_2$ induces a lattice polarization of $Z$ by a lattice ${L}$ of rank $20 - u$, then we expect the mirror $W$ to admit an $\check{L}$-polarized K3 surface fibration, for some lattice $\check{L}$ of rank $u$. Moreover, the space of divisors in $W$ that are supported on fibres of the fibration should have rank $v+1 = h^{2,1}(Y_1) + h^{2,1}(Y_2) + 1$. Note that this is completely consistent with the predictions of Section \ref{sec:threefold}.

\subsection{Threefolds fibred by mirror quartics} \label{sec:mirrorquartics} Our next aim is to demonstrate how this works in a special case: that of threefolds fibred by mirror quartic K3 surfaces. As we will see, in this setting the predictions of classical mirror symmetry, described above, mesh perfectly with the construction presented in Section \ref{sect:setup}.

A detailed study of threefolds fibred by mirror quartic K3 surfaces was conducted in \cite{cytfmqk3s}. We begin by briefly recapping the main construction and results of that paper, before describing how it fits into our picture.

The goal of \cite{cytfmqk3s} is to answer the following question: let $W$ be a Calabi-Yau threefold and assume that $W$ admits a fibration over $\mathbb{P}^1$ by K3 surfaces, $\pi\colon W \rightarrow \mathbb{P}^1$. Assume that the general fibre of $\pi$ is a K3 surface $S$ with  $\mathrm{Pic}(S) = M_2$, for $M_2 := H \oplus E_8 \oplus E_8 \oplus \langle -4 \rangle$ (i.e. $S$ is Dolgachev-Nikulin mirror to a quartic hypersurface in $\mathbb{P}^3$), and that the monodromy representation acts trivially on $M_2$. Such a structure is called an \emph{$M_2$-polarized K3 fibration} on $W$. In \cite{cytfmqk3s} we attempted to classify Calabi-Yau threefolds admitting $M_2$-polarized K3 fibrations.

In order to describe this classification, we start by taking a basic family, called $\mathcal{X}$ in \cite{cytfmqk3s}, which is a smooth resolution of the family of hypersurfaces
\[\{\lambda w^4 + xyz(x+y+z-w)= 0\} \subset \PP^3,\]
for $\lambda \neq 1/256,0$. This defines a smooth family of K3 surfaces over $\mathbb{P}^1 \setminus \{0,1/256,\infty\}$. The classification of $M_2$-polarized fibrations from \cite{cytfmqk3s} can then be stated as:

\begin{theorem} \cite[Section 2]{cytfmqk3s}
If $\pi\colon W \rightarrow \mathbb{P}^1$ is an $M_2$-polarized K3 fibration on a Calabi-Yau threefold $W$, then there is a map
\[g \colon \mathbb{P}^1 \rightarrow \mathbb{P}^1\]
so that $W$ is birational to $g^*\mathcal{X}$ where
\[\xymatrix{
g^*\mathcal{X} \ar[d] \ar[r] &\mathcal{X} \ar[d] \\
U  \ar[r]^(.3){g|_U}       & \mathbb{P}^1\setminus \{0,1/256, \infty\}}\]
and $U = g^{-1}(\mathbb{P}^1 \setminus \{0,1/256,\infty\})$. 

Moreover, the preimage of $0$ under $g$ consists of either 1 or 2 points. If $g^{-1}(0)$ is 2 points, then $g$ is ramified to order $1,2$ or $4$ at each point in $g^{-1}(0)$. If $g^{-1}(0)$ is a single point, then $g$ is ramified to order $8$ at that point. If $g$ is unramified over $1/256$ then $W$ is smooth, but if $g$ ramifies over $1/256$ then $W$ may have isolated singularities.
\end{theorem}

Using this, in \cite[Section 4]{cytfmqk3s} we obtained a classification of smooth deformation equivalent families of such Calabi-Yau threefolds, under the assumption that $g$ is unramified over $1/256$ (which, by the theorem above, ensures that our Calabi-Yau threefolds will be smooth). This classification is determined by two pieces of data: a pair of numbers $i$ and $j$ in $\{1,2,4\}$, denoting the orders of ramification of $g$ over $0$, and a choice of partition $\mu := [x_1,\dots, x_k]$ of $\deg(g) = i+j$, denoting the ramification profile of $g$ over $\infty$. Note here that the case where $g^{-1}(0)$ is a single point is a deformation of the case where $i=j=4$, so may be ignored; see \cite[Remark 3.1]{cytfmqk3s}. We call a general member of this family $W^{\mu}_{i,j}$. 

\begin{theorem} \label{thm:Whodgenumbers} \cite[Propositions 3.5, 3.8]{cytfmqk3s}
If $W^{\mu}_{i,j}$ is as above, then
\begin{enumerate}
\item $h^{2,1}(W^{\mu}_{i,j}) = k$,
\item $h^{1,1}(W^{\mu}_{i,j}) = 20 + \sum_{s =1}^k (2x_{s}^2 + 1) + c_i + c_j$.
\end{enumerate} 
where $c_1 = 30$, $c_2 = 10$ and $c_4 = 0$.
\end{theorem}

Finally, the proof of \cite[Proposition 2.5]{cytfmqk3s} classifies the singular fibres of these threefolds, up to small birational transformations. We give names to each possibility and a description of one member of each equivalence class.
\begin{enumerate}
\item $\mathrm{I}_0$: A smooth K3 surface. Along with the generic fibre, which clearly has type $\mathrm{I}_0$, if $t \in g^{-1}(0)$ and $g$ ramifies to order $4$ at $t$, then the fibre over $t$ is of type $\mathrm{I}_0$.
\item $\mathrm{I}_{\mathrm{odp}}$: A K3 surface with a single node. If $t \in g^{-1}(1/256)$, then the fibre over $t$ is of type $\mathrm{I}_{\mathrm{odp}}$.
\item $\mathrm{I}_n^{\Delta}$  for $n \in \mathbb{N}$: this is a semistable singular fibre whose dual graph is a triangulation of the faces of a $3$-dimensional simplex with sides of length $n$. Such a fibre has $2n^2 +2$ irreducible components. If $t \in g^{-1}(\infty)$ and $g$ ramifies to order $n$ at $t$, then the fibre over $t$ is of type $\mathrm{I}^\Delta_n$.
\item $\mathrm{II}$: A union of $11$ smooth rational surfaces, one of which, $E$, has multiplicity $2$, and the others, $F_1,\dots, F_{10}$, have multiplicity $1$. Each $F_i$ intersects $E$ in a smooth rational curve, but $F_i \cap F_j$ is empty for $i \neq j$. If $t \in g^{-1}(0)$ and $g$ ramifies to order $2$ at $t$, then the fibre over $t$ is of type $\mathrm{II}$.
\item $\mathrm{IV}$: A normal crossings union of $31$ smooth rational surfaces. One has multiplicity $4$, and the rest have multiplicities $3$, $2$ or $1$. If $t \in g^{-1}(0)$ and $g$ is unramified at $t$, then the fibre over $t$ is of type $\mathrm{IV}$.
\end{enumerate}

\subsection{Degenerations in the mirror}

Now we look at the mirror picture. The Dolgachev-Nikulin mirror to an $M_2$-polarized K3 surface is a K3 surface of degree $4$ ($\langle 4 \rangle$-polarized); generically such K3 surfaces are just smooth quartics in $\mathbb{P}^3$. There are three Fano threefolds which contain a K3 surface of degree $4$ as their anticanonical hypersurface; these are $X_1$, the quartic threefold in $\mathbb{P}^4$, $X_2$, the double cover of $\mathbb{P}^3$ ramified along a smooth generic quartic surface, and $X_4$, which is just $\mathbb{P}^3$ itself. Note that the integer $i$ assigned to each Fano $X_i$ is just the index of that Fano threefold. 

Let $Z$ be a generic smooth anticanonical K3 surface in both $X_i$ and $X_j$, for some choice of integers $i,j$ in $\{1,2,4\}$. Then $N_{Z/X_i} \cong \omega_{X_i}^{-1}|_Z \cong \mathcal{O}_Z(i)$, where $\mathcal{O}_Z(1)$ is the restriction of a hyperplane section in $\mathbb{P}^3$ to $Z$. Thus if we take a normal crossings union of $X_i$ and $X_j$ meeting along $Z$ then, in accordance with the discussion in Section \ref{sec:smoothtyur}, we cannot construct a Calabi-Yau threefold by smoothing $X_i \cup_Z X_j$, since $N_{Z/X_i} \otimes N_{Z/X_j} = \mathcal{O}_Z(i+j) \neq \calO_Z$.

This is similar to the situation in Section \ref{sect:Tyur}, where we could not smooth $X_1 \cup_Z X_2$ without first blowing up the locus $V \cap X_1 \cap X_2$, and we will solve it in the same way. Let  $C_1,\dots, C_k$ be smooth curves in $S$ cut out by sections of $\mathcal{O}_Z(x_s)$, for positive integers $x_1,\dots, x_k$, so that $\sum_{s = 1}^k x_s = i+j$. Then let $f\colon \widetilde{X}_i \to X_i$ be the blow up of $X_i$ in $C_1,\dots, C_k$ sequentially and let $E_i$ be the exceptional divisor over $C_i$, for $i=1,\dots, k$. 

The canonical divisor of $\widetilde{X}_i$ is given by $f^*K_{X_i} - \sum_{s=1}^k E_s$, so $\omega_{\widetilde{X}_i}^{-1} \cong \mathcal{O}_Z(-j)$. Therefore, according to \cite[Theorem 4.2]{ldncvsdcyv}, we may smooth $\widetilde{X}_i \cup_Z X_j$ to a Calabi-Yau threefold. We denote this threefold by $V^{\mu}_{i,j}$, where, as before, $\mu$ denotes the partition $[x_1,\ldots,x_k]$ of $(i+j)$. 

We claim that $V^{\mu}_{i,j}$ and $W^{\mu}_{i,j}$ are mirror dual, in the classical sense. As a first piece of evidence for this, we compute the Hodge numbers of $V^{\mu}_{i,j}$.

\begin{proposition}\label{Vhodgenumbersprop}
Let $i,j \in \{1,2,4\}$ and let $\mu = [x_1,\ldots,x_k]$ be a partition of $i+j$. Then the Hodge numbers of the threefold $V_{i,j}^{\mu}$ are given by
\begin{align*}
 h^{1,1}(V_{i,j}^{\mu}) &= k, \\
h^{2,1}(V_{i,j}^{\mu}) &= 20 + \sum_{s =1}^k (2 x_s^2 + 1) + h^{2,1}(X_i) + h^{2,1}(X_j),
\end{align*}
where $h^{2,1}(X_s) = 30$ \textup{(}resp. $10$, $0$\textup{)} for $s = 1$  \textup{(}resp. $2$, $4$\textup{)}.
\end{proposition}
\begin{proof} By definition, $V_{i,j}^{\mu}$ is a smoothing of $\widetilde{X}_i \cup_Z X_j$. Define 
\[q := \rank (\mathrm{im} (H^2(\widetilde{X}_i,\mathbb{Z}) \oplus H^2(X_j,\mathbb{Z}) \rightarrow H^2(Z,\mathbb{Z}))).\]
Then Lee \cite[Corollary 8.2]{cycsncv} shows that the Hodge numbers of $V_{i,j}^{\mu}$ are given by
\begin{align*}
h^{1,1}(V_{i,j}^{\mu}) &= h^2(\widetilde{X}_i) + h^2(X_j) -q-1, \\
h^{2,1}(V_{i,j}^{\mu}) &= 21 + h^{2,1}(\widetilde{X}_i) + h^{2,1}(X_j) - q.
\end{align*}

Now, since the N\'eron-Severi group of $Z$ is generated by the restriction of a hyperplane section from $X_i$, we must have $q=1$. Moreover, since we blew up $X_i$ a total of $k$ times to obtain $\widetilde{X}_i$, we have $h^{2}(\widetilde{X_i}) = k+1$ and $h^2(X_j) = 1$. Thus $h^{1,1}(V_{i,j}^{\mu}) = k$.

To compute $h^{2,1}(V_{i,j}^{\mu})$, we begin by noting that a smooth curve $C_s$ defined by a section of $\mathcal{O}_Z(x_s)$ has self-intersection $4x_s^2$ in $Z$. So the genus formula for curves on a surface gives $g(C_s) = 2x_s^2 + 1$. Thus, by standard results on the cohomology of a blow-up (see, for instance, \cite[Theorem 7.31]{htcagi}), we find
\[h^{2,1}(\widetilde{X}_i) = h^{2,1}(X_i) + \sum_{s=1}^k (2x_s^2 +1),\]
giving the claimed result for $h^{2,1}(V_{i,j}^{\mu})$. Finally, the values of $h^{2,1}(X_s)$ are easy to compute explicitly.
\end{proof}

Putting this proposition together with the result of Theorem \ref{thm:Whodgenumbers}, we obtain:

\begin{corollary}
Let $i,j \in \{1,2,4\}$ be a pair of integers and let $\mu = [x_1,\ldots,x_k]$ be a partition of $i+j$. Then there is a mirror duality between the Hodge numbers of the Calabi-Yau threefolds $V_{i,j}^{\mu}$ and $W_{i,j}^{\mu}$ .
\end{corollary}

We expect that $V_{i,j}^{\mu}$ is actually mirror to $W_{i,j}^{\mu}$, but of course this is not a proof. As further evidence, however, we can also compare filtrations as in Section \ref{sect:MScones}. For the threefolds $V_{i,j}^{\mu}$, we may compute the limit mixed Hodge structure associated to the degeneration to $\widetilde{X}_i \cup_Z X_j$, to obtain 
\[\begin{array}{rcccc}
& \Gr_F^3 & \Gr_F^2 & \Gr_F^1 & \Gr_F^0 \\
\Gr^M_4 &\mathbb{C} & \mathbb{C}^{19} & \mathbb{C} & 0 \\
\Gr^M_3 &0 & \mathbb{C}^v & \mathbb{C}^v &0  \\
\Gr^M_2 & 0 & \mathbb{C} & \mathbb{C}^{19} & \CC
\end{array}\]
for $v =  h^{2,1}(X_i) + h^{2,1}(X_j) + \sum_{s =1}^k (2 x_s^2 + 1)$. 

Now, for the threefolds $W_{i,j}^{\mu}$, let $J$ be the cup product operator with the negative of the class of a fibre of the $M_2$-polarized K3 surface fibration on $W_{i,j}^{\mu}$. Then, by the proof of \cite[Proposition 3.5]{cytfmqk3s} and Proposition \ref{Vhodgenumbersprop}, we see that the rank of the space of divisors in $H^{1,1}(W_{i,j}^{\mu})$ that are supported on fibres is 
\[\rank(\ker(J)) = 1 +\sum_{s =1}^k (2 x_s^2 + 1) + h^{2,1}(X_i) + h^{2,1}(X_j) = v+1.\] 
Moreover, the image of $J$ in $H^{1,1}(W_{i,j}^{\mu})$ is the span of the class of a fibre and the image of $J$ in $H^{3,3}(W_{i,j}^{\mu})$ spans $H^{3,3}(W_{i,j}^{\mu})$. Finally, the image of $J$ in $H^{2,2}(W_{i,j}^{\mu})$ is the space of classes dual to divisors in $H^{1,1}(W_{i,j}^{\mu})$ swept out by monodromy invariant divisors on a general fibre, which has rank $19$. Thus, we obtain
\[\begin{array}{rcccc}
& H^{0,0}(X_{i,j}^\mu) & H^{1,1}(W_{i,j}^{\mu}) & H^{2,2}(W_{i,j}^{\mu}) & H^{3,3}(W_{i,j}^{\mu}) \\
\mathrm{coimage}(J) & \CC & \mathbb{C}^{19} & \mathbb{C} & 0 \\
\ker(J)/ \mathrm{im}(J) &0 & \mathbb{C}^v & \mathbb{C}^v &0  \\
\mathrm{im}(J) & 0 & \mathbb{C} & \mathbb{C}^{19} & \CC
\end{array}\]
and the duality of bifiltered vector spaces discussed in Section \ref{sect:MScones} is satisfied in this case. 

Finally, we note that restriction of divisors from $X_i$ and $X_j$ induces a lattice polarization of $Z$ by the lattice $\langle 4 \rangle$, whilst the K3 surface fibration on $W_{i,j}^\mu$ is $M_2$-polarized. As expected from our previous calculations, these two lattices are Dolgachev-Nikulin mirror dual.

\subsection{Relationship to LG models} 

Now we will see how this fits with the results of Section \ref{sect:setup}. We begin by describing the LG models of $X_1,X_2$ and $X_4$. Again, we note that these threefolds can be constructed from the basic K3 fibred threefold $\mathcal{X}$.

\begin{theorem}\cite{cpmd,mft}
Define $g_{\ell}\colon\mathbb{A}^1 \rightarrow \mathbb{A}^1$ by $g_{\ell}(t) = t^{\ell}$. Then the LG model $(Y_{\ell},\mathsf{w}_{\ell})$ of $X_{\ell}$ is a partial compactification of $g^*_{\ell}{\mathcal{X}}$.
\end{theorem}

In fact, it follows from our classification of singular fibres in Section \ref{sec:mirrorquartics} that 
\begin{enumerate}
\item $(Y_1,\mathsf{w}_1)$ has two singular fibres in $\mathbb{A}^1$, located at $0$ and $1/256$, of types $\mathrm{IV}$ and $\mathrm{I}_\mathrm{odp}$ respectively,
\item $(Y_2,\mathsf{w}_2)$ has three singular fibres in $\mathbb{A}^1$, located at $0$ and $\pm\sqrt{1/256}$, of types $\mathrm{II}$ and $\mathrm{I}_\mathrm{odp}$ respectively,
\item $(Y_4,\mathsf{w}_4)$ has four singular fibres in $\mathbb{A}^1$, located at $\sqrt[4]{1/256}$, all of which have type $\mathrm{I}_\mathrm{odp}$. Its fibre over $0$ has type $\mathrm{I}_0$.
\end{enumerate}

The following ansatz has been suggested by Katzarkov \cite{hmsac}:

\begin{ansatz}
Blowing up a Fano threefold $X$ in a smooth curve of genus $g \geq 2$ has the effect of deforming the LG model $(Y,\mathsf{w})$ of $X$, so that a semistable fibre with $g+1$ irreducible components moves from infinity to a point in $\mathbb{A}^1$.
\end{ansatz}

By this logic, we can induce that if $\widetilde{X}_{\ell}$ is as above, then its LG model $(\widetilde{Y}_{\ell},\widetilde{\mathsf{w}}_{\ell})$ has singular fibres of the types listed above for $(Y_{\ell},\mathsf{w}_{\ell})$, along with fibres of type $\mathrm{I}_{x_{s}}^\Delta$, for ${s} =1,\dots, k$, since a smooth curve $C_{s}$ determined by a section of $\mathcal{O}_Z(x_{s})$ has genus $2x_{s}^2 + 1$. 

By the results of \cite[Section 4]{cytfmqk3s}, generically we may choose the map $g$ determining $W^{\mu}_{i,j}$ so that its ramification points not lying over $0$ and $\infty$ are all of order $2$, and so that all nonzero branch points $\lambda$ of $g$ have $|\lambda| \geq R$ for some real number $R > 1/256$. Then there are two components to the preimage of the disc $U_R = \{ z \in \mathbb{C} : |z| < R\}$, corresponding to the two preimages of $0$. Let $U^{\ell}_R$ be the component of $g^{-1}(U_R)$ containing the ramification point of $g$ of order $\ell$ over $0$ (for $\ell \in \{i,j\}$). Then the restriction of the fibration $\pi$ on $W^{\mu}_{i,j}$ to $U^{\ell}_R$ is a fibration over a disc with a fibre of type $\mathrm{IV}$, $\mathrm{II}$ or $\mathrm{I}_0$ over $g^{-1}(0)$, depending on whether ${\ell} = 1$, $2$ or $4$ respectively, and ${\ell}$ fibres of type $\mathrm{I}_{\mathrm{odp}}$. In fact, this fibration over $U^{\ell}_R$ is deformation equivalent to the LG model of $X_{\ell}$, in the sense that as $R \to \infty$, the map $g$ degenerates to a stable map from a pair of rational curves $C_i$ and $C_j$ meeting at a point, such that the restriction of $g$ to $C_{\ell}$ is $g_{\ell}$. We also see that monodromy around the boundary of $U_R^{\ell}$ is equal to the monodromy around $\infty$ of the the LG model $(Y_{\ell},\mathsf{w}_{\ell})$.

Away from $0$ and $\infty$, we have that the restriction of $\pi$ to $g^{-1}(\mathbb{P}^1 \setminus U_R)$ has singular fibres of type $\mathrm{I}_{x_i}^\Delta$, for $i=1,\dots,k$, and these account for all of the singular fibres of $\pi$ restricted to $g^{-1}(\mathbb{P}^1 \setminus U_R)$. We thus obtain the following theorem; which is highly reminiscent of our philosophy in the K3 surface case (see Section \ref{sec:discussion}).

\begin{theorem}
The threefold $W^{\mu}_{i,j}$ is topologically equivalent to the gluing of the LG model of $X_j$ to the LG model of $\widetilde{X}_i$, as described in Section \ref{sect:glue}.
\end{theorem} 

\section{Non-commutative fibrations}\label{sec:noncomm}

Now we turn our attention to ways in which this construction can fail. Suppose that $V$ and $W$ are mirror dual Calabi-Yau threefolds. The computation in Section \ref{sect:MScones} suggests that, if a Tyurin degeneration of $V$ occurs along a locus in the moduli space of $V$ that contains a maximally unipotent monodromy point, then there should exist a K3 fibration on $W$ whose properties are consistent with the results in Section \ref{sec:threefold}. 

However, it is also possible for Tyurin degenerations to occur along loci in the moduli space of $V$ that are disjoint from points of maximally unipotent monodromy. For such Tyurin degenerations the argument in Section \ref{sect:MScones} will break down and, in particular, in such cases we will have no guarantee of the existence of K3 fibrations on $W$. For an example where this occurs, consider:

\begin{example}\label{ex:cubic}
Let $W$ be the complete intersection of two cubics in $\mathbb{P}^5$ and let $V$ be its Batyrev-Borisov mirror. Since $h^{2,1}(V) = h^{1,1}(W) =1$, the dimension of the moduli space of $V$ is $1$ and, indeed, it can be shown that the moduli space of Calabi-Yau varieties deformation equivalent to $V$ is $\mathbb{P}^1 \setminus \{0,1,\infty\}$. In \cite{ghmsc} Katzarkov and Przyjalkowski show that, after semi-stabilization, the degeneration of $V$ over $\infty$ is a Tyurin degeneration. The degenerate fibre is a union of two quasi-Fano threefolds $X_1$, $X_2$ meeting in a K3 surface $Z$, which has Picard rank 20 and transcendental lattice isomorphic to
\[\left( \begin{matrix}  2& 1 \\ 1 & 2\end{matrix} \right).\]

There are several problems with this example. First of all, we find that this Tyurin degeneration is not connected to a point of maximally unipotent monodromy in the moduli space of $V$, so the discussion above does not hold. An even more grave issue is the fact that the K3 surface $Z$ has no Dolgachev-Nikulin mirror. We would therefore not expect a K3 surface fibration corresponding to this Tyurin degeneration to appear on the mirror Calabi-Yau threefold $W$ and, indeed, the non-existence of such a fibration is easily verified, since $h^{1,1}(W) = 1$.
\end{example}

However, if we are willing to replace the mirror Calabi-Yau threefold $W$ with its bounded derived category of coherent sheaves $\mathbf{D}^b(W)$, we \emph{do} know how to apply mirror symmetry to monodromy in non-maximally unipotent families. Indeed, if $\mathcal{V}$ is a family of Calabi-Yau threefolds over a punctured disc $U$, such that there is a symplectic form $\omega$ on $\mathcal{V}$ which restricts to a symplectic form on each fibre, then the action of monodromy around $0$ induces a symplectomorphism on a smooth fibre $V$. According to \cite{cmhms}, this symplectomorphism induces an autoequivalence on the Fukaya category $\mathcal{F}\mathrm{uk}(V,\omega)$, which passes through mirror symmetry to produce an element of $\mathrm{AutEq}(\mathbf{D}^b(W))$.

Our goal is to understand what sort of autoequivalence this is. If $\pi\colon W\rightarrow \mathbb{P}^1$ is a fibration of $W$ by K3 surfaces, then there is a right derived pullback functor $\mathbf{R}\pi^*\colon \mathbf{D}^b(\mathbb{P}^1) \rightarrow \mathbf{D}^b(W)$. We conjecture that the autoequivalence of $\mathbf{D}^b(W)$ obtained above is an autoequivalence related to the spherical functor $\mathbf{R}\pi^*$. Moreover, if $p$ is a point in $\mathbb{P}^1$, then the inclusion gives a pullback $\mathbf{R} i^*\colon \mathbf{D}^b(\mathbb{P}^1) \rightarrow \mathbf{D}^b(p)$. We can take the fibre product of these categories, denoted $\mathbf{D}^b(W) \otimes_{\mathbf{D}^b(\mathbb{P}^1)} \mathbf{D}^b(p)$, to give a category associated to a general fibre of $W$. We want this category to be the bounded derived category of coherent sheaves on a K3 surface. 

To state this idea more precisely, we need a definition. A triangulated category $\mathbf{T}$ with Serre functor $\mathsf{S}$ is called \emph{$d$-Calabi-Yau} if $\mathsf{S}$ is equivalent to $[d]$, where $[d]$ indicates the natural ``shift by $d$'' functor. This is a useful notion because the bounded derived category $\mathbf{D}^b(X)$ of coherent sheaves on a smooth projective variety $X$ is $d$-Calabi-Yau if and only if $X$ is a Calabi-Yau variety of dimension $d$.

\begin{conjecture}\label{conj:der}
Let $V$ be a Calabi-Yau variety of dimension $d$ that admits a Tyurin degeneration, and let $W$ be a homological mirror of $V$. Then there is a functor $F\colon \mathbf{D}^b(\mathbb{P}^1) \rightarrow \mathbf{D}^b(W)$ so that for a generic point $p \in \mathbb{P}^1$, the category $\mathbf{D}^b(V) \otimes_{\mathbf{D}^b(\mathbb{P}^1)} \mathbf{D}^b(p)$ is a $(d-1)$-Calabi-Yau category.
\end{conjecture}

Moreover, we expect that the fibres of such a categorical fibration are equivalent, in an appropriate sense, to the Fukaya category of the Calabi-Yau $Z = X_1 \cap X_2$, where $X_1 \cup_Z X_2$ is the Tyurin degeneration of $V$.

\begin{example} \label{ex:cubic2} Calabrese and Thomas \cite{decy3fc4f} showed how this should work in the setting of Example \ref{ex:cubic}. Let $f_1$ and $f_2$ be cubics so that $W =\{ f_1 = f_2 = 0\} \subset \PP^5$. Blowing up $\mathbb{P}^5$ along $W$, one obtains a five-fold $\widetilde{\mathbb{P}}^5$ which is fibred over $\mathbb{P}^1$ by cubics $s f_1 + t f_2 = 0$, for $[s : t] \in \mathbb{P}^1$. This variety has derived category
\[\mathbf{D}^b(\widetilde{\mathbb{P}}^5) = \langle \mathbf{D}^b(W), \widetilde{\pi}^*\mathbf{D}^b(\mathbb{P}^1) (i,0) : i = 3, 4, 5 \rangle \]
where $\widetilde{\pi}\colon \widetilde{\mathbb{P}}^5 \rightarrow \mathbb{P}^1$ is the natural map.

Using this, Calabrese and Thomas show that there is a fibration structure on $\mathbf{D}^b(W)$ satisfying the conditions of Conjecture \ref{conj:der} above. In this case the derived categories $\mathbf{D}^b(W) \otimes_{\mathbf{D}^b(\mathbb{P}^1)} \mathbf{D}^b(p)$ are 2-Calabi-Yau categories, but they are not the bounded derived category of coherent sheaves on any K3 or abelian surface; instead they are semi-orthogonal summands of the derived categories of the cubic fourfolds cut out by the equations $s f_1 + t f_2 = 0$.

Moreover, according to Hassett \cite{srcf}, the transcendental lattice of a generic cubic fourfold is isomorphic to 
\[\mathrm{E}_8^2 \oplus \mathrm{U}^2 \oplus  \left( \begin{matrix}  -2& -1 \\ -1 & -2\end{matrix} \right),\]
which is precisely the Dolgachev-Nikulin dual lattice associated to the transcendental lattice of the K3 surface $Z$ in Example \ref{ex:cubic}. We thus conjecture that the non-commutative fibration found by Calabrese and Thomas \cite{decy3fc4f} is mirror dual to the Tyurin degeneration of Katzarkov and Przyjalkowski \cite{ghmsc}.
\end{example}

It may be possible to generate other examples like this one as follows. Suppose that $\mathcal{V}$ is a family of Calabi-Yau threefolds over a small disc $U$, so that the general fibre $V$ has $h^{2,1}(V) = 1$  and the fibre over $0$ is a Tyurin degeneration. Then the matrix describing the monodromy action on $H^3(V,\QQ)$ associated to a loop around $0$ has two Jordan blocks of rank $1$. 

The families of Calabi-Yau variations of Hodge structure lying over $\mathbb{P}^1 \setminus \{0,1,\infty\}$ with appropriate monodromy properties have been classified by Doran and Morgan \cite{doranmorgan}. One can check that there are precisely three families in this classification for which the monodromy matrix associated to $\infty$ has two Jordan blocks of rank $1$. These families are mirrors to
\begin{enumerate}
\item The $(3,3)$ complete intersection in $\mathbb{P}^5$,
\item The $(4,4)$ complete intersection in $\mathbb{WP}(1,1,1,1,2,2)$,
\item The $(6,6)$ complete intersection in $\mathbb{WP}(1,1,2,2,3,3)$.
\end{enumerate}
This raises the following natural question.

\begin{question}
Are there categorical fibrations on $\mathbf{D}^b(W)$ satisfying Conjecture \ref{conj:der}, for $W$ a $(4,4)$ complete intersection in $\mathbb{WP}(1,1,1,1,2,2)$ or a $(6,6)$ complete intersection in $\mathbb{WP}(1,1,2,2,3,3)$?
\end{question}

Neither of these Calabi-Yau threefolds $W$ can admit commutative K3 fibrations, since $h^{1,1}(W) = 1$, so if mirror dual fibrations exist then they are necessarily non-commutative, as in Example \ref{ex:cubic2} above.

\bibliography{Books}
\bibliographystyle{amsalpha}
\end{document}